\documentclass[12pt,oneside]{amsart}
\usepackage[doi=false,isbn=false,url=false,style=alphabetic]{biblatex}
\bibliography{citations.bib}
\usepackage[size=tiny]{todonotes}
\usepackage{mathrsfs}

\usepackage{microtype}

\usepackage{amsthm,amssymb,amsfonts,thmtools}
\usepackage{xcolor}
\definecolor{darkgreen}{rgb}{0,0.30,0}
\definecolor{darkred}{rgb}{0.75,0,0}
\definecolor{darkblue}{rgb}{0,0,0.6}
\usepackage[pdfborder=0,colorlinks,citecolor=darkred,linkcolor=darkred,urlcolor=darkblue]{hyperref}% Autoref and theorems
\usepackage{cleveref}

\def\makeautorefname#1#2{\expandafter\def\csname#1autorefname\endcsname{#2}}

\makeautorefname{ans}{answer}
\makeautorefname{assump}{assumption}
\makeautorefname{cla}{claim}
\makeautorefname{conj}{conjecture}
\makeautorefname{cor}{corollary}
\makeautorefname{cex}{counterexample}
\makeautorefname{code}{codeblock}
\makeautorefname{data}{data}
\makeautorefname{dig}{digression}
\makeautorefname{disc}{discussion}
\makeautorefname{def}{definition}
\makeautorefname{ex}{example}
\makeautorefname{exs}{examples}
\makeautorefname{fac}{fact}
\makeautorefname{intu}{intuition}
\makeautorefname{lem}{lemma}
\makeautorefname{meta}{metathm}
\makeautorefname{nota}{notation}
\makeautorefname{note}{note}
\makeautorefname{prop}{proposition}
\makeautorefname{ques}{question}
\makeautorefname{rmk}{remark}
\makeautorefname{set}{setup}
\makeautorefname{term}{terminology}
\makeautorefname{thm}{theorem}
\makeautorefname{upsh}{upshot}
\makeautorefname{warn}{warning}

\theoremstyle{definition}
\newtheorem{theorem}{Theorem}[section]
\makeautorefname{sec}{Section}
\makeautorefname{subsec}{Subsection}

\newtheorem{assumption}[theorem]{Assumption}

\newtheorem{corollary}[theorem]{Corollary}
\newtheorem{counterexample}[theorem]{Counterexample}

\newtheorem{definition}[theorem]{Definition}
\newtheorem{example}[theorem]{Example}

\newtheorem{lemma}[theorem]{Lemma}

\newtheorem{notation}[theorem]{Notation}
\newtheorem{note}[theorem]{Note}
\newtheorem{proposition}[theorem]{Proposition}

\newtheorem{setup}[theorem]{Setup}
\newtheorem{remark}[theorem]{Remark}
\newtheorem{terminology}[theorem]{Terminology}

\newtheorem{warn}[theorem]{Warning}

\makeatletter
\let\c@corollary=\c@theorem
\let\c@proposition=\c@theorem
\let\c@lemma=\c@theorem
\let\c@assumption=\c@theorem
\let\c@conjecture=\c@theorem
\let\c@definition=\c@theorem
\let\c@example=\c@theorem
\let\c@remark=\c@theorem
\let\c@notation=\c@theorem
\let\c@strategy\c@theorem
\makeatother

\providecommand{\Aut}{\text{Aut}}

\providecommand{\BU}{\text{BU}}

\providecommand{\EU}{\text{EU}}

\providecommand{\Ex}{\text{Ex}}

\providecommand{\GL}{\text{GL}}
\providecommand{\Gr}{\text{Gr}}

\providecommand{\Ho}{\text{Ho}}

\providecommand{\hocolim}{\text{hocolim}}

\providecommand{\id}{\text{id}}
\providecommand{\im}{\text{im}}

\providecommand{\KO}{\text{KO}}

\providecommand{\KU}{\text{KU}}

\providecommand{\MU}{\text{MU}}

\providecommand{\Res}{\text{Res}}
\providecommand{\RO}{\text{RO}}

\providecommand{\SO}{\text{SO}}

\providecommand{\Sym}{\text{Sym}}

\providecommand{\Tr}{\text{Tr}}

\providecommand{\C}{\mathbb{C}}

\providecommand{\P}{\mathbb{P}}

\providecommand{\R}{\mathbb{R}}
\providecommand{\Z}{\mathbb{Z}}
\RequirePackage{mathtools}

\providecommand{\po}{\arrow[ul,phantom,"\ulcorner" very near start]}
\providecommand{\pb}{\arrow[dr,phantom,"\lrcorner" very near start]}
\providecommand{\xto}[1]{\xrightarrow{#1}}

\providecommand{\hookto}{\xhookrightarrow{}}

\makeatletter
\providecommand{\superimpose}[2]{%
  {\ooalign{$#1\@firstoftwo#2$\cr\hfil$#1\@secondoftwo#2$\hfil\cr}}}
\makeatother

\RequirePackage{bbm}

\usepackage{graphicx}

\usepackage{caption}
\usepackage[left=3cm,right=3cm,top=3cm,bottom=3cm]{geometry}

\definecolor{xred}{RGB}{148,2,5}
\definecolor{xdarkblue}{RGB}{7,43,187}
\definecolor{xlightblue}{RGB}{42,157,176}
\definecolor{xviolet}{RGB}{98,70,107}
\definecolor{xgreen}{RGB}{0,116,31}
\definecolor{xyellow}{RGB}{182,143,64}
\definecolor{xpink}{RGB}{221,5,203}
\usepackage{xcolor}

\definecolor{purple}{RGB}{128, 0, 255}

\definecolor{darkgreen}{rgb}{0,0.30,0}
\newcommand{\green}[1]{\color{darkgreen}#1 \color{black}}

\definecolor{darkred}{rgb}{0.75,0,0}

\definecolor{darkblue}{rgb}{0,0,0.6} 

\usepackage{tikz,tikz-cd,pgfplots}
\pgfplotsset{compat=1.18}
\usepackage{tikzit}
\tikzstyle{blue text}=[fill=none, draw=none, shape=circle, text={rgb,255: red,8; green,44; blue,201}]

\tikzstyle{xdb dotted}=[-, draw={rgb,255: red,7; green,43; blue,187}, line width=1pt, dotted]
\tikzstyle{black 0.8}=[-, draw=black, line width=0.8pt]
\tikzstyle{black line}=[-, draw=black, line width=0.25mm]
\tikzstyle{black dashed}=[-, dashed, line width=1pt]

\tikzstyle{xdb line}=[-, draw={rgb,255: red,7; green,43; blue,187}, line width=1pt]
\tikzstyle{xdb dashed}=[-, draw={rgb,255: red,7; green,43; blue,187}, dashed, line width=0.8pt]

\usepackage{wrapfig}
\usepackage{subcaption}
\usepackage{longtable}
\usepackage{multicol}
\usepackage{mdframed}
\usepackage{float}

\let\smashprod\wedge
\let\epsilon\varepsilon

\DeclareMathOperator{\ind}{ind}
\DeclareMathOperator{\Perf}{Perf}

\DeclareMathOperator{\PT}{PT}
\DeclareMathOperator{\Sp}{Sp}
\DeclareMathOperator{\Spc}{Spc}
\DeclareMathOperator{\Th}{Th}
\providecommand{\EU}{\text{EU}}

\renewcommand{\P}{\mathbb{P}}

\providecommand{\CP}{\mathbb{C}\text{P}}
\providecommand{\RP}{\mathbb{R}\text{P}}

\let\til\widetilde
\providecommand{\green}[1]{\textcolor{green}{#1}}
\providecommand{\red}[1]{\textcolor{red}{#1}}

\providecommand{\green}[1]{\textcolor{green}{#1}}
\providecommand{\blue}[1]{\textcolor{blue}{#1}}

\providecommand{\spherespec}{\mathbb{S}}
\providecommand{\algK}{\mathrm{K}}

\providecommand{\GOS}[1]{\mathrm{Sp}_{G}^{#1}}
\newcommand{\Ret}[2]{\mathrm{Ret}_{#1}^{#2}}

\newcommand\noloc{%
   \nobreak
   \mspace{6mu plus 1mu}
   {:}
   \nonscript\mkern-\thinmuskip
   \mathpunct{}
   \mspace{2mu}
}

\DeclareFieldFormat{postnote}{\mknormrange{#1}}
\DeclareFieldFormat{volcitepages}{\mknormrange{#1}}
\DeclareFieldFormat{multipostnote}{\mknormrange{#1}}

\title{Equivariant enumerative geometry}

\author{Thomas Brazelton}
\date{\today}

\begin{document}

\begin{abstract}
We formulate an \textit{equivariant conservation of number}, which proves that a generalized Euler number of a complex equivariant vector bundle can be computed as a sum of local indices of an arbitrary section. This involves an expansion of the Pontryagin--Thom transfer in the equivariant setting. We leverage this result to commence a study of enumerative geometry in the presence of a group action. As an illustration of the power of this machinery, we prove that any smooth complex cubic surface defined by a symmetric polynomial has 27 lines whose orbit types under the $S_4$-action on $\CP^3$ are given by $[S_4/C_2]+[S_4/C_2'] + [S_4/D_8]$, where $C_2$ and $C_2'$ denote two non-conjugate cyclic subgroups of order two. As a consequence we demonstrate that a real symmetric cubic surface can only contain 3 or 27 real lines.
\end{abstract}

\maketitle

\setcounter{tocdepth}{1}
\tableofcontents{}

\section{Introduction}
Enumerative geometry poses geometric questions of the form ``how many?''~and expects integral answers. Over two millennia ago Apollonius asked how many circles are tangent to any three generic circles drawn on the plane. In the mid-1800's Salmon and Cayley famously proved that there are 27 lines on a smooth cubic surface over the complex numbers, and it is a classical result that there are 2,875 lines on a general quintic threefold. The power of enumerative geometry lies in the principle of \textit{conservation of number} --- that enumerative answers are conserved under changes in initial parameters: there are eight circles tangent to \textit{any} three generic circles, 27 lines on \textit{any} smooth cubic surface, and 2,875 lines on \textit{any} general quintic threefold.

In this work we propose solving enumerative problems in the presence of a group action. Related works include \cite{Damon,Bethea,C2Bezout}, but these differ in perspective. We formulate and prove a version of conservation of number in this context which allows us to compute answers to equivariant enumerative problems valued in the Burnside ring of a group.

\begin{theorem}\label{thm:main-conservation} \textit{(Equivariant conservation of number)} Let $G$ be any finite group, and let $p\colon E \to M$ be an equivariant complex vector bundle of rank $n$ over a smooth proper $G$-equivariant $n$-manifold, and let $A$ be a complex oriented $\RO(G)$-graded cohomology theory. Let $\sigma\colon M \to E$ be any equivariant section with isolated simple zeros. Then we have a well-defined Euler number valued in $\pi_0^G A$, computed by:
\begin{align*}
    n(E) = \sum_{G\cdot x \subseteq Z(\sigma)} \Tr_{G_x}^G(1).
\end{align*}
Working in homotopical complex bordism $\MU_G$, as a corollary we may see that given two such sections $\sigma,\sigma'$, there is an isomorphism of $G$-sets between their zero loci $Z(\sigma) \cong Z(\sigma')$ (\autoref{thm:cons-number}). In other words, \textit{the $G$-action on the solutions to such an enumerative problem is conserved.}
\end{theorem}
Our result is more general, admitting local indices for more general zero loci than isolated simple points (see \autoref{lem:conservation}), however the context stated above is sufficient to carry out some computations.

To illustrate the power of this machinery, consider the case of a smooth cubic surface $X = V(F) \subseteq \CP^3$. We will say that $X$ is $S_4$-\textit{symmetric} (or just \textit{symmetric} for short) if it is fixed under the $S_4$-action on $\CP^3$ by permuting coordinates (equivalently, $F(x_0,x_1,x_2,x_3)$ is a symmetric homogeneous polynomial). We know classically that there are 27 lines on $X$, however under the $S_4$-action lines on $X$ are mapped to other lines on $X$. It is natural then to inquire whether the $S_4$-orbits of the lines on $X$ are conserved as the symmetric cubic surface varies. It turns out that this question admits an answer that doesn't depend upon the choice of $S_4$-symmetric cubic surface.

\begin{theorem}\label{thm:intro-27} On \textit{any} smooth symmetric cubic surface over the complex numbers, the 27 lines come in the following orbits:
\begin{align*}
    [S_4/C_2] + [S_4/C_2'] + [S_4/D_8],
\end{align*}
where $C_2$ and $C_2'$ are two non-conjugate subgroups of $S_4$ of order two. Explicitly, there are 12 lines in an orbit with isotropy group $C_2 = \left\langle (1\ 2) \right\rangle$, 12 lines in an orbit with isotropy group $C_2' = \left\langle (1\ 3)(2\ 4) \right\rangle$, and three lines in an orbit with isotropy group $D_8$.
\end{theorem}

On the famous \textit{Clebsch cubic surface}, which is symmetric, all 27 lines are defined over the reals, and we can visualize their orbits in \autoref{fig:clebsch}.
\begin{figure}[h]
  \includegraphics[width=0.3\linewidth]{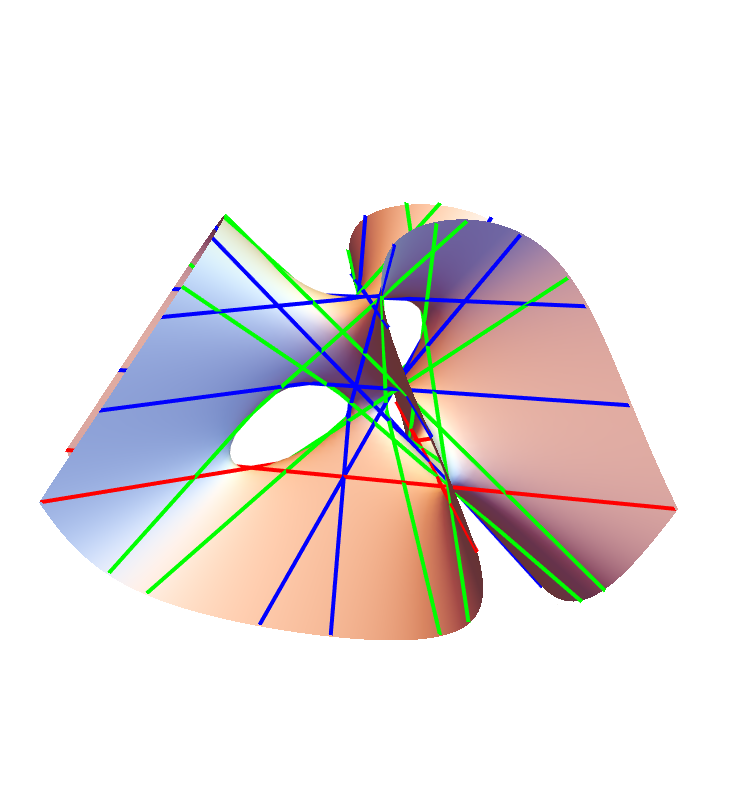}
  \includegraphics[width=0.3\linewidth]{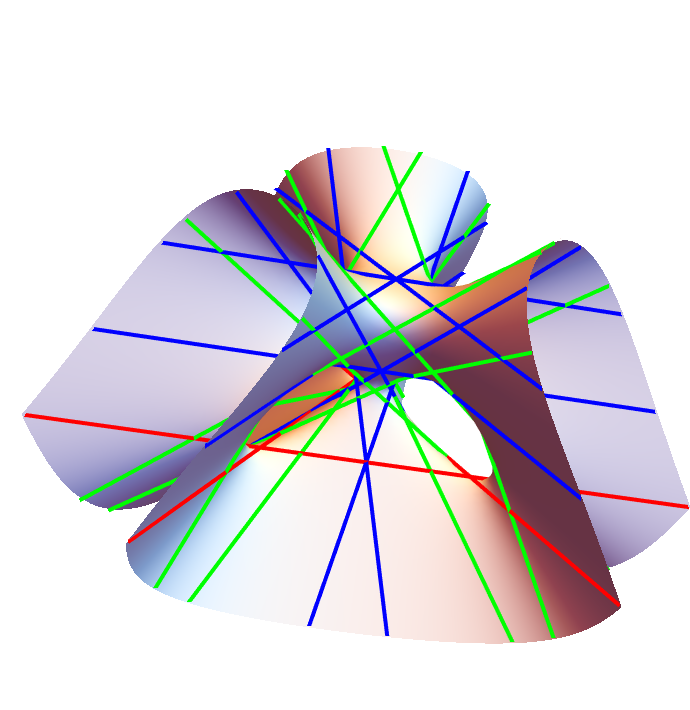}
  \includegraphics[width=0.3\linewidth]{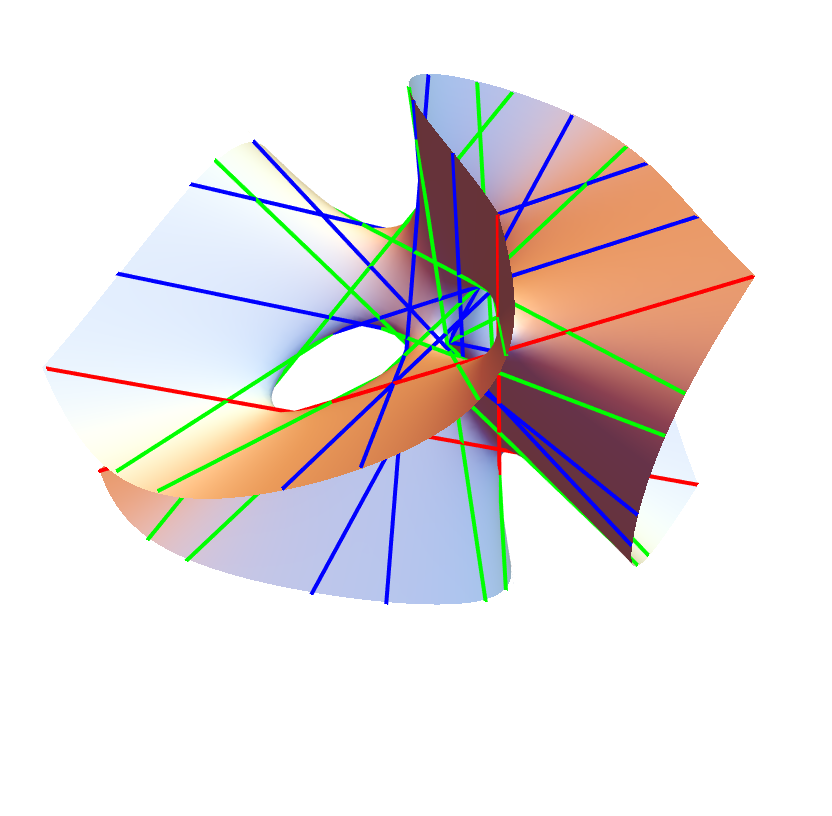}
  \centering
  \caption{The 27 lines on the Clebsch surface, grouped into $S_4$-orbits according to color, pictured from a few different angles. An animated version is available on the author's webpage.}
  \label{fig:clebsch}
\end{figure}

Given a real cubic surface, as it admits 27 lines after base changing to the complex numbers it is natural to ask how many of these are defined over the reals. Schl\"afli's Theorem tells us that a smooth real cubic surface can contain 3, 7, 15, or 27 lines, and all of these possibilities do occur \cite{Schlafli}. Under the presence of symmetry we can refine this result.

\begin{theorem}\label{thm:intro-27-real} A smooth real symmetric cubic surface can contain either 3 or 27 lines, and both of these possibilities do occur.
\end{theorem}

\subsection{Outline} In \autoref{sec:GOS} we discuss the theory of equivariant retractive spaces and parametrized spectra. We extend the theory of duality as laid out in \cite{Hu}, and discuss dualizing objects in terms of cotangent complexes. This allows us to define Thom transformations analogous to those found in the motivic setting, and to flesh out the six functors formalism for genuine orthogonal parametrized $G$-spectra.

In \autoref{sec:PT}, we explore the Pontryagin--Thom transfer from \cite{Hu,MS,ABG}. These will induce Gysin maps that allow us to push forward cohomology classes and carry out computations in the equivariant parametrized setting.

In \autoref{sec:cohomology} we provide a broad definition of compactly supported equivariant cohomology, twisted by a perfect complex, valued in any genuine ring spectrum. This culminates in the important result that, under certain orientation assumptions, cohomology classes twisted by a vector bundle can be pushed forward and expressed as sums of local contributions coming from the components of the zero locus of a section of a bundle.

In \autoref{sec:conservation} we discuss refined Euler classes in the parametrized equivariant setting. We recap the theory of equivariant complex orientations, and state and prove equivariant conservation of number (\autoref{thm:main-conservation}, as \autoref{thm:cons-number}). We use this to commence a study of enumerative geometry in the equivariant setting.

In \autoref{sec:27}, we provide an application of equivariant conservation of number by investigating the orbits of the 27 lines on a smooth symmetric cubic surface and proving that they are independent of the choice of symmetric cubic surface (\autoref{thm:intro-27}, as \autoref{thm:27-lines}). We argue that both the field of definition and the topological type (\autoref{def:type-of-line}) of a line are preserved under the group action. This allows us to eliminate the possibilities of seven or 15 real lines on a real symmetric cubic, refining Schl\"afli's Theorem in the symmetric setting (\autoref{thm:intro-27-real}, as \autoref{thm:27-real}).

\subsection{Acknowledgements} Thank you to Mona Merling and Kirsten Wickelgren for inspiring and supporting this work, to Cary Malkiewich for countless conversations about details big and small, to Markus Hausmann for helpful correspondence about equivariant orientations, and to William Balderrama for communicating equivariant nilpotence to us. Thank you also to Candace Bethea, Benson Farb, Joe Harris, Mike Hopkins, Sidhanth Raman, Frank Sottile, and Zhong Zhang for fruitful conversation about this work and related topics. Finally thank you to the anonymous referee whose comments greatly improved this paper. The author is supported by an NSF Postdoctoral Research Fellowship (DMS-2303242).

\section{Retractive spaces and parametrized spectra, equivariantly}\label{sec:GOS}
In this section we will establish technical machinery with the ultimate goal of obtaining well-defined Euler numbers for equivariant sections of complex bundles over smooth proper $G$-manifolds. In direct analogy to the theory of Euler classes in motivic homotopy theory, we will want to work over a base space, i.e. in a parametrized way. This yields many advantages, including a more streamlined characterization of dualizing objects, access to a natural six functors formalism, and a clear discussion of how $\RO(G)$-graded cohomology extends to $\KO_G(X)$-graded cohomology.

An advantage of working in this setting is the presence of \textit{Thom transformations}, which we will define as certain auto-equivalences in the stable setting. Explicitly, when working over a $G$-manifold $M$, we can take an equivariant vector bundle $E \to M$, and smash over $M$ with the fiberwise Thom space $\Th_M(E)$. We use these transformations to define twisted cohomology classes valued in any genuine ring spectrum, and develop the theory of their pushforwards. In particular we will see that we have a well-defined \textit{Euler class}, which pushes forward to an \textit{Euler number} in complex oriented cohomology theories. This number can be interpreted, and will serve as our main tool for solving equivariant enumerative problems.

\begin{notation} (Categorical notation)
Throughout, when working with categories, a subscript will denote a group $G$, following a convention in equivariant homotopy theory for working with genuine $G$-equivariance, while a superscript will denote that we are working parametrized over a space $X$. For example $\Sp_G^X$ will denote the category of genuine orthogonal $G$-equivariant parametrized $X$-spectra (\autoref{nota:GOSX}). When a superscript is omitted we working non-equivariantly, i.e. with the trivial group, and when a subscript is omitted we are working parametrized over a point.
\end{notation}

\begin{remark} (On machinery): We work here with the model category $\GOS{X}$ of genuine orthogonal $G$-spectra parametrized over a $G$-space or spectrum $X$. The reader should be warned that, should they venture deeper into this category, they may encounter some point-set issues obstructing a true 1-categorical six functors formalism, e.g. the pushforward $f_\ast$ is homotopically poorly behaved \cite[2.2.15]{Cary}, and even may fail to preserve weakly Hausdorff spaces \cite[\S2.2]{MS}, \cite{Lewis}. Thankfully none of these issues will rear their heads in this particular work.
\end{remark}

\begin{assumption} All spaces and all maps will be assumed to be equivariant with respect to the action of a compact Lie group $G$ unless otherwise explicitly stated. If not otherwise clarified, any statements made about vector bundles are equally true for both real and complex bundles.
\end{assumption}

\subsection{Basic definitions}

Given a $G$-space $X$, we denote by $(\Spc_G)_{/X}$ the slice category of $G$-spaces equipped with an equivariant map to $X$. This category isn't pointed, so we cannot make sense of phenomena like suspension and thus stabilization. To rectify this, we slice it under $X$ in order to obtain the category $\Ret{G}{X} := (\Spc_G)_{X//X}$ of retractive $G$-spaces over $X$. More explicitly:

\begin{definition} The category $\Ret{G}{X}$ of \textit{retractive} $G$\textit{-spaces over} $X$ has as objects commutative diagrams of the form
\[ \begin{tikzcd}
    X\rar\ar[dr,"\id" below left] & Y\dar\\
     & X.
\end{tikzcd} \]
That is, the category of spaces which equivariantly retract onto $X$. The morphisms are equivariant maps $Y\to Y'$ which commute with the inclusion and projection maps.
\end{definition}

\begin{example} $\ $
\begin{itemize}
    \item The category of retractive $G$-spaces over a point $\Ret{G}{}$ is the category of based $G$-spaces.
    \item For any subgroup $H\le G$, there is an equivalence of categories $\Ret{G}{G/H} \simeq \Ret{H}{}$.
\end{itemize}
\end{example}

\begin{proposition} The category $\Ret{G}{X}$ has finite products and coproducts --- given $A,B\in \Ret{G}{X}$ the product is given by the pullback $A \times_X B$ along the projection maps, while the coproduct is the pushout $A\cup_X B$ along the inclusions. We denote the coproduct by $A\vee_X B := A\cup_X B$ to stress the comparison with the wedge product in pointed spaces.
\end{proposition}

\begin{example}\label{ex:plus-space} Let $Y$ be any $G$-space equipped with a map $f\colon Y\to X$. Denote by $Y_{+X} \in \Ret{G}{X}$ the retractive space $Y \amalg X$, with inclusion given by mapping $X$ to itself, and projection given by $f$ and the identity:
\[ \begin{tikzcd}
    X\rar\ar[dr,"\id" below left] & Y \amalg X \dar["f \amalg \id" right]\\
     & X.
\end{tikzcd} \]
\end{example}

\begin{definition} Given $A,B \in \Ret{G}{X}$, we define their \textit{fiberwise smash product} $A \smashprod_X B$ to be the fiberwise cofiber of the map between the coproduct and product:
\[ \begin{tikzcd}
    A\vee_X B\rar\dar & A\times_X B\dar\\
    X\rar & A\smashprod_X B.
\end{tikzcd} \]
This turns $\Ret{G}{X}$ into a symmetric monoidal category, with unit given by $S^0_X := X_{+X}$.
\end{definition}

\begin{example} More generally, any $G$-representation $V$ has an associated \textit{representation sphere} $S^V$, which is the one-point compactification, based at the point at infinity. Denote by $S^V_X = X \times S^V$ the \textit{fiberwise representation sphere}. This has a natural projection to $X$, and by convention the fiber over $x$ is based at the point at infinity in $S^V$.
\end{example}

\begin{example} If $p\colon E \to X$ is a $G$-equivariant vector bundle, then the zero section endows it with the structure of an $X$-retractive space.
% \[ \begin{tikzcd}
%     X\rar["s"]\ar[dr,"\id_X" below left] & E\dar["p" right]\\
%      & X.
% \end{tikzcd} \]
\end{example}

\begin{example} Given an equivariant vector bundle $p\colon E \to X$, denote by $\Th_X(E)$ the \textit{fiberwise Thom space}, where the fibers $E_x$ have each been compactified to a different point at infinity (one obtains the ordinary Thom space $\Th(E)$ by gluing these points at infinity together):
\[ \begin{tikzcd}
    S^{E_x}\rar\dar\pb & \Th_X(E)\dar\\
    \left\{ x \right\}\rar[hook] & X.
\end{tikzcd} \]
The action of $G$ on $X$ induces an action on the points at infinity, and the inclusion of $X$ into these points at infinity endows $\Th_X(E)$ with the structure of a retractive $G$-space.
\end{example}

\begin{definition}\label{def:forgetful-functor-retractive-spaces} Let $f\colon X \to Y$ be a $G$-map. There is a \textit{forgetful} functor
\begin{align*}
    f_\sharp \colon \Ret{G}{X} \to \Ret{G}{Y},
\end{align*}
given by sending a retractive space $S$ over $X$ to the pushout $S\cup_f Y$ with inclusion and projection maps induced by the pushout:
\[ \begin{tikzcd}
    X\rar["f" above]\dar & Y\dar\ar[ddr, bend left=10,"\id" above right]\\
    S\rar\ar[dr, bend right=5] & f_\sharp S\po\ar[dr,dashed] \\
    & X\rar["f" below] & Y.
\end{tikzcd} \]
\end{definition}

\begin{warn} There is competing notation for the six functors appearing in parametrized homotopy theory, so we should clarify our notational choices before proceeding. May and Sigurdsson \cite{MS} refer to the functor described in \autoref{def:forgetful-functor-retractive-spaces} as $f_!$. We use $f_\sharp$ for this functor, as does \cite{Hu}, and we reserve the shriek notation for the \textit{exceptional adjunction}, which we will define in \autoref{def:exceptional-adjunction}.
\end{warn}

\begin{notation} For any $G$-space $X$, denote by $\pi_X \colon X \to \ast$ the unique map to a point.
\end{notation}

\begin{example}\label{ex:forgetting-along-structure-map} Given a $G$-equivariant vector bundle $E \to X$, applying $(\pi_X)_\sharp$ to the fiberwise Thom space of a vector bundle $\Th_X(E)$ has the effect of collapsing the basepoint copy of $X$, meaning that it glues all the points at infinity together to one. This recovers the ordinary (non-fiberwise) Thom space $\Th(E)$. As a particular case, the trivial bundle $X \times V \to V$ has fiberwise Thom space the fiberwise representation sphere $S^V_X$. The pushforward $(\pi_{X})_{\sharp}S^V_X$ is the half-smash product $X_+ \smashprod S^V$.
\end{example}

\begin{definition} Any $G$-map $f\colon X \to Y$ induces a \textit{pullback} functor
\begin{align*}
    f^\ast \colon \Ret{G}{Y} \to \Ret{G}{X},
\end{align*}
given by sending a retractive space $T$ over $Y$ to the pullback $f^\ast T$ with inclusion and projection maps induced by the pullback:
\[ \begin{tikzcd}
    X\ar[dr,dashed]\ar[drr,bend left=10,"\id" above right]\dar["f" left] & & \\
    Y\ar[dr, bend right=5] & f^\ast T \pb\rar\dar & X\dar["f" right]\\
    & T\rar & Y.
\end{tikzcd} \]
\end{definition}

\begin{proposition} The pullback functor is symmetric monoidal. In particular we observe that $\pi_X^\ast S^V = S^V_X$.
\end{proposition}

\begin{proposition} It is straightforward to check that there is an adjunction $f_\sharp \dashv f^\ast$.
\end{proposition}

\subsection{Model structures on parametrized $G$-spectra}

We endow the category of retractive $G$-spaces with the $q$\textit{-model structure}, which first appeared in the non-equivariant setting in an unpublished preprint of May \cite{May-ex-spaces}, was fleshed out in the equivariant setting by Hu \cite{Hu}, and was recently made explicit by Malkiewich \cite{Cary}.

Let $f \colon S\to T$ be a map in $\Ret{G}{X}$. We say it is a \textit{weak equivalence} if it is a weak equivalence when viewed as a morphism in $\Spc_G$; explicitly, if $f^H \colon S^H \to T^H$ is a weak homotopy equivalence for every subgroup $H \le G$. Similarly we define $f$ to be a \textit{fibration} if $f^H$ is a Serre fibration for every $H \le G$. Cofibrations in this model structure are given by retracts of relative $G$-cell complexes.

\begin{notation}\label{nota:trivial-bundle} Given any $G$-space $X$, and any complex $G$-representation $V$, denote by $\epsilon_X^V$ the $G$-vector space $X \times V \to X$. We call this the \textit{trivial} bundle associated to the representation $V$.
\end{notation}

\begin{definition} A \textit{genuine orthogonal $G$-spectrum} over $X$ is a sequence of $X$-retractive spaces $A_V$ for each real orthogonal representation $V$, together with the data of structure maps
\begin{align*}
    \epsilon_X^{W-V} \smashprod_X A_V \to A_W
\end{align*}
for every linear $G$-equivariant isometric inclusion $V \hookto W$.
\end{definition}

\begin{notation}\label{nota:GOSX} 
We denote by $\GOS{X}$ the category of genuine orthogonal $G$-spectra parametrized over $X$ (denoted $G\mathcal{O}\mathcal{S}(X)$ in \cite{Cary}). The fiberwise smash product on retractive spaces induces a fiberwise smash product on $\GOS{X}$, and this category is also symmetric monoidal with unit given by the sphere spectrum $\spherespec_X \in \GOS{X}$. When we write $\mathbb{S}$ without a subscript, it is understood we are working fiberwise over a point, so we obtain the equivariant sphere spectrum in $\Sp_G$, but with the group notation suppressed.
\end{notation}

\begin{example} For $S\in \Ret{G}{X}$, we denote by $\Sigma_X^\infty S$ the \textit{suspension spectrum}, with component spaces $S_V = \epsilon_X^V \smashprod_X S$. This is a functorial procedure, giving an adjunction
\begin{align*}
    \Sigma_X^\infty \colon \Ret{G}{X} \leftrightarrows \GOS{X} \noloc \Omega^\infty_X.
\end{align*}
\end{example}

Given $f \colon X \to Y$ the definitions of $f^\ast$ and $f_\sharp$ extend to genuine orthogonal spectra by applying them levelwise (c.f. \cite[\S4.3]{Cary}), yielding an adjunction
\begin{equation}\label{eqn:upperstar-pushf-adjunction}
\begin{aligned}
    f_\sharp \colon  \GOS{X} \leftrightarrows \GOS{Y} \noloc  f^\ast.
\end{aligned}
\end{equation}
Just as in the unstable setting, the pullback $f^\ast$ is symmetric monoidal.

The $q$-model structure we outlined for $\Ret{G}{X}$ can be extended to a model structure on $\GOS{X}$ by defining weak equivalences and fibrations componentwise \cite[Definition~3.3]{Hu}. This forms a closed model structure \cite[Proposition~3.4]{Hu}, \cite[Theorem~1.0.1]{Cary}. We denote by $[-,-]_X$ homotopy classes of maps in $\GOS{X}$.

\begin{proposition} \cite[\S3]{Hu} The adjunction $f_\sharp \dashv f^\ast$ in \autoref{eqn:upperstar-pushf-adjunction} is a Quillen adjunction. 
\end{proposition}

\subsection{Projection and exchange}

Two key techniques frequently used in settings where a six functors formalism appears are a \textit{projection formula} and \textit{exchange transformations}. Projection describes the interaction of the forgetful functor with smash products, while exchange describes how functors channel data from commutative diagrams of base objects (in this case retractive $G$-spaces).

\begin{theorem}\label{thm:projection} \textit{(Projection)} \cite[4.7]{Hu} Let $f\colon X \to Y$ be a $G$-map of spaces, and take $S\in \GOS{X}$ and $T \in \GOS{Y}$. Then there is an isomorphism in $\GOS{Y}$, which is natural in both $S$ and $T$:
\begin{align*}
    T \smashprod_Y f_\sharp(S) \xto{\sim} f_\sharp \left( f^\ast T \smashprod_X S \right).
\end{align*}
\end{theorem}

\begin{example}\label{ex:sharp-pullback-smash} In the case where $S=S^0_X$ is the zero-sphere over $X$, projection takes the form
\begin{align*}
    T \smashprod_Y f_\sharp(S^0_X) \xto{\sim} f_\sharp f^\ast(T).
\end{align*}
That is, applying $f_\sharp f^\ast(-)$ has the effect of smashing fiberwise with $f_\sharp(S^0_X)$.
\end{example}

\begin{theorem} \textit{(Exchange)} \cite[2.2.11]{MS},~\cite[2.2.11]{Cary} For any commutative square of $G$-spaces
\begin{equation}\label{eqn:square-G-spaces}
\begin{aligned}
    \begin{tikzcd}[ampersand replacement=\&]
    A\rar["f"]\dar["g" left] \& B\dar["q" right] \\
    C\rar["p" below] \& D,
\end{tikzcd}
\end{aligned}
\end{equation}
there is an associated \textit{exchange transformation} $\Ex_\sharp^\ast \colon g_\sharp f^\ast \to p^\ast q_\sharp$. This is a weak equivalence if the square is homotopy cartesian.
\end{theorem}

\subsection{Thom transformations}

If $A\in\Sp$, then its $n$th cohomology groups are defined by $A^n(X) = [X,\Sigma^n A]$, which we may write less concisely as $\left[ -, \Th(\R^n) \smashprod A \right]$ by considering $S^n$ to be the Thom space of a rank $n$ bundle over a point. Passing to the parametrized setting over a base space $X$, it might make sense then, for any vector bundle $E \to X$, to define the $E$\textit{th cohomology group} $A^E(-)$ by $\left[ -, \Th_X(E) \smashprod_X A \right]_{X}$. We will make such a definition in \autoref{sec:cohomology}, but first we explore the process of smashing fiberwise with a Thom space of a vector bundle. This will let us define an invertible endofunctor on $\GOS{X}$ which we call a \textit{Thom transformation}.

\begin{definition}\label{def:thom-transformation-vector-bundle} Let $E \to X$ be a $G$-equivariant vector bundle. We define the associated \textit{Thom transformation}, denoted by $\Sigma_X^E$, to be the endofunctor defined by smashing fiberwise with the fiberwise Thom space of $E$.
\begin{align*}
    \Sigma_X^E : \GOS{X} &\to \GOS{X} \\
    S &\mapsto S \smashprod_X \Th_X(E).
\end{align*}
\end{definition}

\begin{example}\label{ex:thom-transformation-trivial-vb} The simplest example is when $E = \epsilon_X^V$ is the trivial bundle associated to any $G$-representation $V$. Applying the associated Thom transformation yields
\begin{align*}
    \Sigma_X^{\epsilon_X^V}(-) &= \Th_X(\epsilon_X^V) \smashprod_X (-) = (X \times S^V) \smashprod_X (-) =  S^V_X \smashprod_X(-).
\end{align*}
That is, it is the same as suspending by the parametrized $V$-sphere over $X$, which is invertible in the world of parametrized $G$-spectra over $X$. We will use $\Sigma_X^V$ instead of the more cumbersome notation $\Sigma_X^{\epsilon_X^V}$.
\end{example}

\begin{proposition}\label{prop:Thom-splits-SES} The Thom transformations are additive, in the sense that for any short exact sequence of equivariant bundles over $X$:
\begin{align*}
    0 \to A \to B \to C \to 0,
\end{align*}
there is an isomorphism $\Sigma_X^A \Sigma_X^C \cong \Sigma_X^B$, which is unique in the homotopy category.
\end{proposition}
\begin{proof} Since every short exact sequence of equivariant bundles is split, we have an isomorphism $B\cong A \oplus C$, inducing a homeomorphism $\Th_X(B) \cong \Th_X(A) \smashprod_X \Th_X(C)$ (c.f.  \cite[6.2.2]{Cary}). As the choice of such splittings forms a contractible space, it is clear that this isomorphism is well-defined up to homotopy.
\end{proof}

Thom transformations are invertible on $\GOS{X}$ for trivial bundles, which relies on the fact that equivariant vector bundles admit stable inverses.

\begin{proposition}\label{prop:complementary-bundle} \cite[2.4]{Segal} Let $E\to X$ be an equivariant vector bundle. Then there is a representation $V$ and a $G$-bundle $E^\perp \to X$ so that $E \oplus E^\perp \cong \epsilon_X^V$.\footnote{In \cite{Segal} this result is stated only for complex vector bundles, but the same argument found there works for real vector bundles by picking a $G$-equivariant Riemannian metric for $E$ in $\epsilon_X^V$ and defining $E^\perp$ to be its complement.}
\end{proposition}

\begin{proposition}\label{prop:inverse-thom-transform} Let $E \to X$ be a $G$-bundle. Then the Thom transformation $\Sigma_X^E$ admits an inverse at the level of the homotopy category $\Sigma_X^{-E}$, which is defined by $\Sigma_X^{-V}\circ \Sigma_X^{E^\perp}$, for any trivial bundle $\epsilon^V_X$ and complementary bundle $E \oplus E^\perp \cong \epsilon^V_X$.
\end{proposition}
\begin{proof} This follows by combining \autoref{prop:Thom-splits-SES} and \autoref{prop:complementary-bundle}.
\end{proof}

With this notion of Thom transformations associated to virtual bundles, we can extend the definition of Thom transformations to hold for perfect complexes over our domain space. 

\begin{notation} We denote by $\KO_G(X)$ the abelian group of isomorphism classes of virtual real $G$-vector bundles over $X$. Observe there is a natural inclusion $\RO(G) \to \KO_G(X)$ given by sending a representation to its associated trivial bundle.
\end{notation}

\begin{corollary} The Thom transformations induce a group homomorphism from the group of isomorphism classes of virtual complex vector bundles over $X$:
\begin{align*}
    \KO_G(X) &\to \Aut(\Ho(\GOS{X})) \\
    [E] &\mapsto \Sigma_X^E.
\end{align*}
\end{corollary}
Following Segal \cite[\S3]{Segal}, we define a \textit{complex of $G$-vector bundles} on $X$ to be a sequence of $G$-vector bundles $E_i$ and equivariant vector bundle maps over $X$:
\[
\cdots \xto{d} E_n \xto{d} E_{n-1} \xto{d} \cdots
\]
so that $d^2 = 0$. We say that a complex $E_\bullet$ is \textit{bounded} if 
$E_n = 0$ for $|n|$ sufficiently large. Let $\Perf(\KO_G(X))$ denote the category of \textit{perfect complexes}, meaning those which are quasi-isomorphic to bounded ones. The following definition is inspired by the motivic $J$-homomorphism of \cite{BachmannHoyois}.
\begin{proposition}\label{prop:thom-transform-on-perf} The Thom transformations extend to perfect complexes of vector bundles on $X$:
\begin{align*}
    \Sigma_X^{(-)} \colon \Perf(\KO_G(X)) &\to \Aut(\Ho(\GOS{X})) \\
    \left( \cdots\to E_n \to E_{n-1} \to \cdots \to E_0 \right) &\mapsto \Sigma_X^{(-1)^n E_n} \circ \cdots \circ \Sigma_X^{E_0}.
\end{align*}
\end{proposition}
\begin{proof} It will suffice to show that the definition above is well-defined on quasi-isomorphism classes of bounded complexes. Suppose $f_\bullet \colon A_\bullet \to B_\bullet$ is a quasi-isomorphism of complexes. Considering the differential $d_n^A \colon A_n \to A_{n-1}$, we have a short exact sequence
\begin{align*}
    0 \to \ker(d_n^A) \to A_n \to \im(d_n^A) \to 0,
\end{align*}
which by \autoref{prop:Thom-splits-SES} induces an isomorphism
\begin{align*}
    \Sigma_X^{(-1)^n A_n} \cong \Sigma_X^{(-1)^{n+1}\ker(d_n^A)} \Sigma_X^{(-1)^{n+1}(\im(d_n^A))}.
\end{align*}
Since $\Sigma_X^{\ker(d_n^A) - \im(d_{n+1}^A)} \cong \Sigma_X^{H_n(A)}$, we observe that
\begin{align*}
    \Sigma_X^{A_\bullet} \cong \sum_n \Sigma_X^{(-1)^{n+1} H_n(A)}.
\end{align*}
As $A$ and $B$ are quasi-isomorphic, we conclude that $\Sigma_X^{A_\bullet}\cong \Sigma_X^{B_\bullet}$.
\end{proof}

\begin{corollary}\label{cor:thom-tr-infinity-functor} $\Perf(\KO_G(X))$ is a stable $\infty$-category, hence we can construct its connective algebraic $K$-theory $\algK(\Perf(\KO_G(X)))$. Any path (zig-zag of quasi-isomorphisms) in this space between two points $E_\bullet,F_\bullet\in\Perf(\KO_G(X))$ induces a canonical natural equivalence $\Sigma_X^{E_\bullet} \cong \Sigma_X^{F_\bullet}$. Hence \autoref{prop:thom-transform-on-perf} can be thought of as an $\infty$-categorical group homomorphism.
\end{corollary}

Finally, we discuss how Thom transformations behave along base change.

\begin{proposition}\label{prop:thom-tr-pullback-pushf} Given $f \colon X \to Y$ and $\xi\in\Perf(\KO_G(Y))$, we have weak equivalences which are natural in $\xi$:
\begin{align*}
    \Sigma_Y^\xi f_\sharp &\xto{\sim} f_\sharp \Sigma_X^{f^\ast \xi}\\
    f^\ast \Sigma_Y^\xi &\xto{\sim} \Sigma_X^{f^\ast \xi} f^\ast.
\end{align*}
\end{proposition}
\begin{proof} We may assume without loss of generality that $\xi$ is a vector bundle over $Y$, since this equivalence will directly extend to perfect complexes. In that case, the first map is an example of the purity theorem \autoref{thm:projection}. The second map is defined to be the mate of the first (after swapping the sign on $\xi$), and we observe it is a natural equivalence.
\end{proof}

\subsection{Cotangent complexes and duality} One of the key constructions in \cite{Hu} is that of a dualizing object $C_f$ associated to a class of morphisms in $\Spc_G$ called \textit{smooth proper families of $G$-manifolds}. We discuss cotangent complexes for both closed immersions and smooth proper families in the language of the Thom transformation of a \textit{cotangent complex} $\mathcal{L}_f \in \Perf(\KO_G(X))$.

\begin{definition}\label{def:smooth-proper-family} A $G$-map $f\colon X \to Y$ is said to be a \textit{smooth proper family of $G$-manifolds} if the fiber over every point is a smooth proper $G$-manifold, varying continuously over $Y$. Here ``proper'' means that the homotopy fibers are compact \cite{ABG}.
\end{definition}

\begin{remark} Duality for parametrized spectra can be checked fiberwise, in the sense that a parametrized $X$-spectrum is dualizable if and only if its fiber over every point in the base is a dualizable spectrum (e.g. a finite spectrum) \cite[Lemma~4.2]{ABG}. The conditions in \autoref{def:smooth-proper-family} imply that $f_\sharp \spherespec_Y$ will be an invertible spectrum over $Y$, and the analogous statement is true equivariantly \cite{Hu}.
\end{remark}

\begin{definition}\label{def:smoothable-proper} We define a map of smooth compact $G$-manifolds $f\colon X \to Y$ to be \textit{smoothable proper} if it admits a factorization
\begin{equation}\label{eqn:lci-factorization}
\begin{aligned}
    \begin{tikzcd}[ampersand replacement=\&]
    X\rar[hook,"i" above]\ar[dr,"f" below left] \& W\dar["\pi" right]\\
     \& Y,
\end{tikzcd}
\end{aligned}
\end{equation}
where $i$ is a closed $G$-embedding and $\pi$ is a smooth proper family of $G$-manifolds.
\end{definition}

Given such a factorization, consider the following two short exact sequences, the first of bundles over $X$ and the second of bundles over $W$:
\begin{equation}\label{eqn:bundles-SES}
\begin{aligned}
   0 \to TX \to i^\ast TW \xto{(1)} Ni \to 0 \\
    0 \to T\pi \to TW \to \pi^\ast TY \to 0. 
\end{aligned}
\end{equation}
Since $i^\ast$ is exact, we can apply $i^\ast$ to the second sequence to obtain
\begin{equation}\label{eqn:bundles-SES-2}
\begin{aligned}
    0 \to i^\ast T\pi \xto{(2)} i^\ast TW \to f^\ast TY \to 0.
\end{aligned}
\end{equation}
This yields a composite
\begin{align*}
    i^\ast T\pi \xto{(1)\circ (2)} Ni.
\end{align*}

\begin{definition} Let $f\colon X  \to Y$ be smoothable proper with factorization $f = \pi\circ i$. Define the \textit{cotangent complex} of $f$ to be the two term complex
\begin{align*}
    \mathcal{L}_f := \left(\cdots \to 0  \to i^\ast T\pi \to Ni\right),
\end{align*}
where $i^\ast T\pi$ is in degree zero and $Ni$ in degree negative one.
\end{definition}

\begin{proposition}\label{prop:cotangent-complex-well-defined} The cotangent complex yields a Thom transformation $\Sigma_X^{\mathcal{L}_f}$ associated to any smoothable proper morphism $f$, which gives a well-defined functor on the homotopy category.
\end{proposition}
\begin{proof} Given any factorization as in \autoref{eqn:lci-factorization}, we may use the short exact sequences in \autoref{eqn:bundles-SES} and \autoref{eqn:bundles-SES-2} to derive equations in $\KO_G(X)$:
\begin{align*}
    [i^\ast T\pi] &= [i^\ast TW] - [f^\ast TY] \\
    [Ni] &= [i^\ast TW] - [TX].
\end{align*}
From this we may observe that the class of the cotangent complex can be described of the difference $[TX] - [f^\ast TY]$. In other words, there is an isomorphism
\begin{align*}
    \Sigma_X^{i^\ast T\pi} \Sigma_X^{-Ni} \cong \Sigma_X^{TX} \Sigma_X^{f^\ast TY}.
\end{align*}
This provides a model of the Thom transformation of the cotangent complex which is independent of the choice of factorization.
\end{proof}

\begin{example} The Thom transformation associated to the projection map $\pi_M \colon M \to \ast$, where $M$ is any smooth compact manifold, is $\Sigma_M^{\mathcal{L}_{\pi_M}} \cong \Sigma_M^{TM}$.
\end{example}

\begin{remark} For a smoothable proper morphism $f\colon X \to Y$, the invertible spectrum $\Sigma_X^{\mathcal{L}_f} \spherespec_X$ is its associated \textit{dualizing object}. In the setting where $f\colon X \to Y$ is a smooth proper family of $G$-manifolds, $\Sigma_X^{\mathcal{L}_f} \spherespec_X$ agrees with Hu's dualizing object $C_f$ as hinted at in the discussion \cite[pp.42---43]{Hu}, where $C_f = \Sigma_X^{Tf} \spherespec_X$ is the Thom space of the relative tangent bundle $Tf = T_{X/Y}$. An illuminating discussion illustrating this example was laid out in \cite[\S4.3]{ABG}.
\end{remark}

\begin{example}\label{ex:cotangent-complexes} Let $f\colon X \to Y$ be a closed $G$-embedding. Then its dualizing object is the fiberwise Thom space of its inverse normal bundle $\Th_X\left(-Nf\right)$.
\end{example}

\begin{example} Let $s\colon X \to E$ denote the zero section of a vector bundle. By \autoref{ex:cotangent-complexes} its cotangent complex is $Ns[-1]$, and we see that its normal bundle is precisely $E$, so its dualizing object is $\Sigma_X^{-E}\spherespec_X$.
\end{example}

\begin{proposition}\label{prop:cotangent-composite} Let $f \colon X \to Y$ and $g \colon Y \to Z$ be two composable smoothable proper morphisms. Then there is a natural isomorphism of functors from $\Ho(\GOS{X})$ to $\Ho(\GOS{Z})$:
\begin{align*}
    \Sigma_X^{\mathcal{L}_{g\circ f}} \cong \Sigma_X^{\mathcal{L}_f} \circ \Sigma_X^{f^\ast \mathcal{L}_g}.
\end{align*}
Here $f^\ast \mathcal{L}_g$ is defined by pulling back the two-term chain complex $\mathcal{L}_g$ along $f$.
\end{proposition}
\begin{proof} We observe that there is a distinguished triangle
\begin{align*}
    \mathcal{L}_f \to \mathcal{L}_{g\circ f} \to f^\ast \mathcal{L}_g.
\end{align*}
This yields a path in $K$-theory, which induces a canonical weak equivalence by \autoref{cor:thom-tr-infinity-functor}.
% Using the formula for the Thom transformation of the cotangent complex as in the proof of \autoref{prop:cotangent-complex-well-defined}, we may write
% \begin{align*}
%     \Sigma_X^{\mathcal{L}_{g\circ f}} \cong \Sigma_X^{TX} \circ \Sigma_X^{-f^\ast g^\ast TZ} \cong \left( \Sigma_X^{TX} \circ \Sigma_X^{- f^\ast TY} \right)\circ \left( \Sigma_X^{f^\ast TY} \circ \Sigma_X^{- f^\ast g^\ast TZ} \right).
% \end{align*}
% From this the result follows.
\end{proof}

\subsection{The exceptional adjunction}

Using cotangent complexes and their associated Thom transformations, we can build the exceptional adjunction.

\begin{definition}\label{def:exceptional-adjunction} Let $f\colon X \to Y$ be smoothable proper. Define the \textit{exceptional functors} by
\begin{align*}
    f_! := f_\sharp \Sigma_X^{- \mathcal{L}_f} \colon \GOS{X} \leftrightarrows \GOS{Y} \noloc  \Sigma_X^{\mathcal{L}_f} f^\ast =: f^!
\end{align*}
\end{definition}

It is direct from the definition that these define adjoint functors.

\begin{proposition}\label{prop:open-embedding} If $f\colon X \to Y$ is an open embedding of smooth $G$-manifolds, then the cotangent complex is trivial, hence $f^\ast \simeq f^!$ and $f_\sharp \simeq f_!$
\end{proposition}
\begin{proof} We note that an open embedding is a smooth proper family of $G$-manifolds.
Since the embedding is open, its differential is an isomorphism, and therefore its relative tangent bundle vanishes.
\end{proof}

\begin{proposition}\label{prop:shrieks-compose} Let $f \colon X \to Y$ and $g \colon Y \to Z$ be smoothable proper. Then there is a natural isomorphism $(g\circ f)^! \cong f^! g^!$, and hence also $(g\circ f)_! \cong g_! f_!$.
\end{proposition}
\begin{proof} Using \autoref{prop:cotangent-complex-well-defined}, we may expand $(g\circ f)^!$ as
\begin{align*}
    (g\circ f)^! &= \Sigma_X^{\mathcal{L}_{g\circ f}} f^\ast g^\ast \cong \Sigma_X^{\mathcal{L}_f} \Sigma_X^{f^\ast \mathcal{L}_g} f^\ast g^\ast.
\end{align*}
Commuting $f^\ast$ past the Thom transformation of the cotangent complex for $g$ via \autoref{prop:thom-tr-pullback-pushf}, we obtain
\begin{align*}
    \Sigma_X^{\mathcal{L}_f} f^\ast \Sigma_X^{\mathcal{L}_g} f^\ast g^\ast = f^! g^!.
\end{align*}
The desired equivalence for the exceptional pushforward follows then from the calculus of mates.
\end{proof}

We end this discussion by recalling the main duality theorem of \cite{Hu}.

\begin{theorem}\label{thm:hu} (\cite[4.9]{Hu}) When $f \colon X \to Y$ is a smooth proper family of compact $G$-manifolds, we obtain a Quillen adjunction:
\begin{align*}
    f^\ast \colon \GOS{Y} \leftrightarrows \GOS{X} \noloc f_!
\end{align*}
\end{theorem}

\section{Equivariant Pontryagin--Thom transfers}\label{sec:PT}

Given a map $f\colon X \to Y$, cohomology classes on $Y$ can be pulled back to classes on $X$. A foundational question in mathematics is when cohomological data can be transmitted the other way.

\begin{example} \textit{(Atiyah duality)} Clasically, given a smooth compact manifold $M$, pushing forward cohomological data along the map $\pi_M\colon M \to \ast$ amounts to integrating cohomology classes in order to produce a scalar. By embedding $M$ in Euclidean space $\R^n$ and then taking a one-point compactification, we obtain an embedding $i\colon M \hookto S^n$. By collapsing $S^n$ onto a tubular neighborhood of the embedding, we obtain, up to diffeomorphism, the Thom space of the normal bundle of the embedding $S^n \to \Th(Ni)$. Desuspending by $n$ gives us a map of spectra $\spherespec \to (\pi_M)_! \spherespec_M$, which we think of as the \textit{dual} of the map $M \to \ast$. Under the presence of a Thom isomorphism, the cohomology of $\Th(-TM)$ agrees with the cohomology of $M$ up to a shift, hence cohomology classes on $M$ can be pulled back along this dual map to cohomology classes of the sphere spectrum. We call the map $\spherespec \to \left( \pi_M \right)_! \spherespec_M$ a \textit{transfer} (also called an \textit{Umkehr map}), and the induced map on cohomology a \textit{Gysin map}.
\end{example}

In this section we explore transfers in the parametrized equivariant setting --- first along closed immersions, second along smooth proper families of $G$-manifolds, and finally developing a key result about composites of transfers working over a point, which will help us develop our theory of pushforwards of equivariant Euler classes.

We begin with a brief recollection about the meaning of duality for parametrized equivariant spectra.

\subsection{Duality for parametrized spectra}

Parametrized spectra come equipped with two natural notions of duality: being \textit{fiberwise duality} and \textit{Costenoble--Waner duality}. Viewing a space $X$ as an $\infty$-category (e.g. by taking its associated fundamental $\infty$-groupoid), we can consider a parametrized spectrum as an $\infty$-functor $F \colon X \to \Sp$, which is equivalently a parametrized spectrum by straightening and unstraightening. Such a functor defines a \textit{fiberwise dualizable} spectrum if $F(x)$ is dualizable for each $x\in X$. It is \textit{Costenoble--Waner dualizable} if the entire assembled spectrum $\hocolim_{x\in X}F(x)$ is dualizable.

We may alternatively view parametrized spectra as a bicategory, where the hom-category between spaces $A$ and $B$ is $\Sp^{A \times B}$. From that perspective, the category of $X$-parametrized spectra can be considered as the hom-category $\Sp^{X \times \ast} \cong \Sp^X$ from $X$ to a point. The \textit{right dual} recovers fiberwise duality, while the \textit{left dual} recovers Costenoble--Waner duality. For further discussion from this perspective, see \cite[Chapter~17]{MS}.

\begin{example} Let $M$ be a smooth compact manifold, and consider the sphere spectrum $\mathbb{S}_M \in \Sp^{M \times \ast}$. Its left (Costenoble--Waner) dual is $\Th_M(-TM) \in \Sp^{\ast \times M}$ and its right (fiberwise) dual is $\mathbb{S}_M$.
\end{example}

\begin{example} If $f \colon E \to M$ is a smooth proper family over a smooth compact manifold, and $f_\sharp \mathbb{S}_E$ is considered as living in $\Sp^{M \times \ast}$, then its Costenoble--Waner dual is $f_\sharp \Th_E(-TE)$, whereas its left dual is $f_\sharp \Th_E(-Tf) = f_! \mathbb{S}_E$.
\end{example}

\begin{example} In the category of spectra, considered as the hom-category $\Sp^{\ast \times \ast}$ in the bicategory of parametrized spectra, fiberwise and Costenoble--Waner duality coincide.
\end{example}

Any map of spaces $f \colon X \to Y$ gives rise to a natural map of parametrized $Y$-spectra $f_\sharp \mathbb{S}_X \to \mathbb{S}_Y$. In the setting where both spaces are dualizable on the same side, we can examine the relevant dual and it often gives rise to a transfer of some sort -- our main examples being a transfer constructed by May and Sigurdsson using Costenoble--Waner duality for closed embeddings \cite[18.6.5]{MS} and the Pontraygin--Thom transfer, which Ando, Blumberg, and Gepner characterize via fiberwise dualizability \cite{ABG}. In order to establish a small case of functoriality for transfers, we will leverage the straightforward functoriality of these natural maps, together with the composite of dual pairs theorem.

\subsection{Transfers along closed immersions}

Let $i\colon Z \hookto X$ be a closed $G$-embedding of smooth compact $G$-manifolds. We will discuss a Pontryagin--Thom transfer of the form $\PT(i)\colon\spherespec_X \to i_! \spherespec_Z$. In the non-equivariant setting, this transfer was constructed by May and Sigurdsson \cite[18.6.3]{MS} (see also \cite[4.17]{ABG}).

First we will better understand the spectrum $i_! \spherespec_X$. An explicit point-set model will be important later as we will leverage it to define refined Euler classes.

\begin{proposition}\label{prop:kw2} Let $i\colon Z \hookto X$ be a closed $G$-embedding. Then there is a weak equivalence in $\GOS{X}$ of the form
\begin{align*}
    i_! \spherespec_Z \simeq \Sigma^\infty C_X(X,X-Z),
\end{align*}
where $C_X(X,X-Z)$ denotes the double mapping cylinder obtained by gluing the cylinder $(X-Z) \times[0,1]$ to two copies of $X$, based at the bottom copy of $X$, with $G$-action happening levelwise in each slice of the cylinder.
\end{proposition}
\begin{proof} By \cite[7.2]{KW2}, there is a weak equivalence in $\Ret{G}{X}$ of the form
\[
    i_\sharp \Th_Z(Ni) \simeq C_X(X,X-Z).
\]
By taking suspension spectra, we would like to see that $\Sigma^\infty i_\sharp \Th_Z(Ni) \simeq i_! \spherespec_Z$. That is, we must demonstrate an equivalence
\begin{align*}
    i_\sharp \Sigma^\infty \Th_Z(Ni) \cong \Sigma^\infty i_\sharp \Th_Z(Ni).
\end{align*}
This follows from a more general fact -- that we need not spectrify when applying the forgetful functor to suspension spectra. This is a natural consequence of projection \autoref{thm:projection}. If $T\in \Ret{G}{X}$ is a retractive $X$-space, then the projection formula yields the following natural isomorphism (the second follows from pullback preserving spheres)
\begin{align*}
    \epsilon_X^{V'-V} \smashprod_X i_\sharp \left( \epsilon^V_Z \smashprod_Z T \right) \cong i_\sharp \left( i^\ast \epsilon_X^{V'-V} \smashprod_Z \epsilon^V_Z \smashprod_Z T \right)\cong i_\sharp \left( \epsilon_Z^{V'} \smashprod_Z T \right).
\end{align*}
In other words, we have that $\left\{ i_\sharp \Sigma_Z^{V} T \right\}_{V\in \RO(G)}$ is already a spectrum.% Observing that $Ni = - \mathcal{L}_i$, the result follows.
\end{proof}

\begin{proposition}\label{prop:PT-closed-inclusion} (\cite[18.6.5]{MS}) Given a closed $G$-embedding of compact $G$-manifolds $i \colon Z \hookto M$, the natural map $i_\sharp \mathbb{S}_Z \to \mathbb{S}_M$ is Costenoble--Waner dualizable, with dual given by applying $\Sigma_M^{TM}$ to a \textit{Pontryagin-Thom transfer}
\begin{align*}
    \PT(i) \colon \mathbb{S}_M \to i_! \mathbb{S}_Z.
\end{align*}
\end{proposition}
The explicit construction of this transfer relies heavily on the equivariant tubular neighborhood theorem, and we refer the reader to \cite[18.6.3]{MS} for details about the construction. Here we will be content with the dualizability of the natural inclusion map of spheres, and the existence of a transfer.

\subsection{Transfers along smooth proper families}

If $f \colon X \to Y$ is a smooth proper family, then via the adjunction in \autoref{thm:hu} we have a natural transformation $\id \to f_! f^\ast$. The component of this transformation at the sphere spectrum is of the form $\PT(f)\colon\spherespec_Y \to f_! \spherespec_X$. This is what is referred to as the \textit{equivariant Pontryagin--Thom transfer} associated to a smooth proper family of $G$-manifolds. If $G$ is the trivial group, this is consistent with the definition found in \cite[4.13]{ABG}. An explicit model for this transfer may be found in \cite{Hu,MS,ABG}.

\begin{example}\label{ex:PT-to-point} If $M$ is a smooth compact $G$-manifold, then the structure map to a point $\pi_M \colon M \to \ast$ induces $(\pi_M)_\sharp \mathbb{S}_M \to \mathbb{S}$ in the category of spectra. Its fiberwise dual and Costenoble--Waner dual coincide, and they are equal to $\PT(\pi)$, which is the classical Pontryagin--Thom collapse map:
\begin{align*}
    \PT(\pi_M) \colon \mathbb{S} \to (\pi_M)_! \spherespec_M = \Th(-TM).
\end{align*}
\end{example}

Suppose we have a closed $G$-immersion of smooth compact $G$-manifolds $Z \hookto M$, and consider the commutative diagram
\[ \begin{tikzcd}
    Z\rar[hook,"i" above]\ar[dr,"\pi_Z" below left] & M\dar["\pi_M" right]\\
     & \ast
\end{tikzcd} \]
An important lemma for us to establish is that transferring along $\pi_Z$ is equivalent to first transferring along $i$ and then along $\pi_M$. This is an immediate consequence of the characterization of transfers as Costenoble--Waner duals to natural maps.

\begin{lemma}\label{lem:PT-compatibility} Let $i \colon Z \hookto M$ be a closed $G$-immersion of smooth compact $G$-manifolds. Then the composite
\begin{align*}
    \spherespec \xto{\PT(\pi_M)} (\pi_M)_! \spherespec_M \xto{(\pi_M)_! \PT(i)} \pi_! i_! \spherespec_Z 
\end{align*}
is weakly equivalent to $\PT(i)$.
\end{lemma}
\begin{proof} It is clear that the composite of natural maps
\begin{align*}
    (\pi_M)_\sharp i_\sharp \mathbb{S}_Z \xto{(1)} (\pi_M)_\sharp \mathbb{S}_M \xto{(2)} \mathbb{S}
\end{align*}
is equivalent to the map $(\pi_Z)_\sharp \mathbb{S}_Z \to \mathbb{S}$ via the natural isomorphism $(\pi_M)_\sharp i_\sharp\cong (\pi_Z)_\sharp$. The Costenoble--Waner dual of the composite is $\PT(\pi_Z)$ by \autoref{ex:PT-to-point}. Via the composite of dual pairs theorem \cite[16.5.1]{MS}, this is equal to the composite of the Costenoble--Waner duals of the two maps. The Costenoble--Waner dual of the map (2) is $\PT(\pi_M)$, so in order to prove the lemma it suffices to verify that the Costenoble--Waner dual of the map (1) is $\pi_! \PT(i)$.

The compatibility of $(\pi_M)_\sharp$ with Costenoble--Waner duality can be found in \cite[17.3.3]{MS}, so it suffices to apply $(\pi_M)_\sharp$ to the Costenoble--Waner dual of $i_\sharp \mathbb{S}_Z \to \mathbb{S}_M$. This dual is $\Sigma_M^{TM}\PT(i)$ by \autoref{prop:PT-closed-inclusion}. Hence altogether the dual of (1) is $(\pi_M)_\sharp\Sigma_M^{TM}\PT(i)$, which is $(\pi_M)_! \PT(i)$.
\end{proof}

\section{Cohomology}\label{sec:cohomology}

Here we develop a theory of cohomology with compact supports, twisted by perfect complexes. This theory mirrors that found in the motivic setting (c.f. \cite{DJK,deloop2,BW3}, etc.). The main goal is to demonstrate that cohomology classes can be pushed forward by forgetting support, or by decomposing along the clopen components of the support. In this sense, certain abstract cohomology classes can be understood in rings as sums of local contributions of data. In \autoref{sec:conservation} we will leverage this perspective to prove conservation of number in the equivariant setting.

\subsection{Twisted cohomology}

Let $\xi \in \Perf(\KO_G(X))$ be a perfect complex of equivariant vector bundles over $X$, and let $A\in\Sp_G$ be an arbitrary genuine $G$-spectrum, which represents an $\RO(G)$-graded cohomology theory. 

\begin{definition}\label{def:twisted-cohomology} Define $\xi$\textit{-twisted cohomology with coefficients in} $A$ by
\begin{align*}
    A^\xi(X) := \left[ \spherespec_X, \Sigma_X^\xi \pi_X^\ast A \right]_{X}.
\end{align*}
\end{definition}

When $\xi$ is a trivial bundle, we show that \autoref{def:twisted-cohomology} recovers $\RO(G)$-indexed cohomology groups.

\begin{example}\label{ex:A-cohomology-twisted-by-triv-bundle} If $\xi = \epsilon_X^V$ is a trivial bundle for some $G$-representation $V$, then $\epsilon_X^V$-twisted cohomology is of the form
\begin{align*}
    A^{\epsilon_X^V}(X) &= \left[ \spherespec_X, S_X^V \smashprod_X \pi_X^\ast A \right]_X = \left[ \left( \pi_X \right)_\sharp S_X^{-V}, A \right] = \left[ X_+ \smashprod S^{-V}, A \right].
\end{align*}
This last group is precisely the definition of $A^V(X)$, that is, the $A$-cohomology of $X$ indexed over $\RO(G)$ (see e.g. \cite[p.~35]{LMS}).
\end{example}

\begin{notation}\label{nota:twist-by-rep} For $V$ a $G$-representation and $A\in\Sp_G$ any spectrum, \autoref{ex:A-cohomology-twisted-by-triv-bundle} indicates that we can use $A^V(X)$ to refer to classical $V$th $A$-cohomology group of $X$ or the $A$-cohomology of $X$ twisted by the trivial vector bundle $\epsilon^V_X$ without loss of generality. Similarly to \autoref{ex:thom-transformation-trivial-vb}, we will freely use $A^V(X)$ instead of $A^{\epsilon_X^V}(X)$.
\end{notation}

When $Z \subseteq X$ is a closed $G$-subspace, we can talk about cohomology classes that are ``supported'' on $Z$. Let $i \colon Z \hookto X$ denote the inclusion map.

\begin{definition}\label{def:twisted-cohom-supports} For $\xi\in \Perf(\KO_G(X))$, define $\xi$\textit{-twisted cohomology with coefficients in $A$ and support on} $Z$ to be
\begin{align*}
    A_Z^\xi(X) := \left[ i_! \spherespec_Z, \Sigma_X^\xi \pi_X^\ast A \right]_X.
\end{align*}
\end{definition}

We should provide some intuition as to why this is a reasonable definition of cohomology supported on $Z$. Recall by \autoref{prop:kw2} that $i_! \spherespec_Z$ is equivalent to the double mapping cylinder $C_X(X,X-Z)$. Collapsing this space along its cylinder coordinate, we obtain the space in \autoref{fig:fried-egg}. 

\begin{figure}[h]
\centering
\resizebox{0.3\columnwidth}{!}{%
\begin{tikzpicture}
	\begin{pgfonlayer}{nodelayer}
		\node [style=none] (8) at (-3.775, 1) {};
		\node [style=none] (9) at (-3.325, 0.6) {};
		\node [style=none] (10) at (-2.85, 0.675) {};
		\node [style=none] (11) at (-2.3, 0.65) {};
		\node [style=none] (12) at (-1.9, 1) {};
		\node [style=none] (13) at (-3.7, 1.525) {};
		\node [style=none] (14) at (-2.925, 1.45) {};
		\node [style=none] (15) at (-1.95, 1.55) {};
		\node [style=none] (16) at (-4.25, 0.25) {};
		\node [style=none] (17) at (-1.4, 0.25) {};
		\node [style=none] (18) at (-3, 1.25) {};
		\node [style=none] (19) at (-2.875, 2.025) {};
		\node [style=none] (20) at (-1.65, 0.525) {\tiny$X$};
		\node [style=blue text] (21) at (-2.9, 0.975) {\tiny\color{xdarkblue}{$Z$}};
	\end{pgfonlayer}
	\begin{pgfonlayer}{edgelayer}
		\draw [style=xdb dotted, in=90, out=0, looseness=0.50] (18.center) to (12.center);
		\draw [style=xdb line, in=0, out=-90, looseness=1.25] (12.center) to (11.center);
		\draw [style=xdb line, in=0, out=180, looseness=1.50] (11.center) to (10.center);
		\draw [style=xdb line, in=0, out=-180, looseness=1.25] (10.center) to (9.center);
		\draw [style=xdb line, in=-90, out=180] (9.center) to (8.center);
		\draw [style=xdb dotted, in=180, out=90, looseness=0.75] (8.center) to (18.center);
		\draw [style=black line, in=150, out=-180] (13.center) to (16.center);
		\draw [style=black 0.8, in=-165, out=-30, looseness=0.50] (16.center) to (17.center);
		\draw [style=black line, in=-15, out=15] (17.center) to (15.center);
		\draw [style=black dashed, in=0, out=165] (15.center) to (14.center);
		\draw [style=black dashed, in=0, out=-180] (14.center) to (13.center);
		\draw [style=xdb line] (8.center)
			 to [in=-180, out=90] (19.center)
			 to [in=90, out=0, looseness=1.25] (12.center);
	\end{pgfonlayer}
\end{tikzpicture}
}
\caption{The homotopy type of the space $i_! S^0_Z$.}
  \label{fig:fried-egg}
\end{figure}
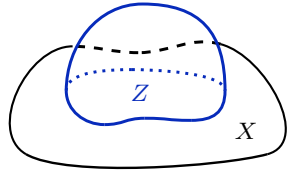

The bottom copy of $X$ is the basepoint, which has to be sent to the basepoint in the target. What we are left with is an extra copy of $Z$, glued along the base, which is free to be mapped anywhere in the target. Thus we think of maps out of $C_X(X,X-Z)$ yielding cohomology classes supported on $Z$.

\begin{definition}\label{def:forget-supports} Precomposition with the Pontryagin--Thom transfer $\PT(i) \colon \spherespec_X \to i_! \spherespec_Z$, defined in \autoref{prop:PT-closed-inclusion}, induces a \textit{forgetting support} map
\begin{align*}
    A_Z^{\xi}(X) \to A^\xi(X).
\end{align*}
\end{definition}

\begin{proposition}\label{prop:twisted-tangent-bundle} Let $M$ be a smooth proper $G$-manifold, and $i: Z\hookto M$ a closed $G$-embedding. Then there is a canonical isomorphism
\begin{align*}
    A_Z^{TM}(M) \cong A^{TZ}(Z).
\end{align*}
\end{proposition}
\begin{proof} We can write
\begin{align*}
    A_Z^{TM}(M) = \left[ i_! \spherespec_Z, \Sigma_M^{TM} \pi_M^\ast A \right]_X \cong \left[ \spherespec_Z, i^! \Sigma_M^{TM} \pi_M^\ast A \right]_Z.
\end{align*}
As the exceptional pullback is given by $i^! = \Sigma_Z^{\mathcal{L}_i} i^\ast = \Sigma_Z^{-Ni} i^\ast$, we may rewrite the above as
\begin{align*}
    \left[ \spherespec_Z, \Sigma_Z^{-Ni} i^\ast \Sigma_M^{TM} \pi_M^\ast A \right]_Z.
\end{align*}
Commuting $i^\ast$ with the Thom transformation via \autoref{prop:thom-tr-pullback-pushf} yields
\begin{align*}
    \left[ \spherespec_Z, \Sigma_Z^{-Ni} \Sigma_Z^{i^\ast TM} \pi_Z^\ast A \right]_Z.
\end{align*}
From the short exact sequence
\begin{align*}
    0 \to TZ \to i^\ast TM \to Ni \to 0,
\end{align*}
we have that $\Sigma_Z^{-Ni} \Sigma_Z^{i^\ast TM} \cong \Sigma_Z^{TZ}$, from which we can see that $A^{TM}_Z(M)$ is isomorphic to
\begin{equation*}
\begin{aligned}
    \left[ \spherespec_Z, \Sigma_Z^{TZ} \pi_Z^\ast A \right]_Z = A^{TZ}(Z).
\end{aligned}\qedhere
\end{equation*}
\end{proof}
Recall classically that compactly supported cohomology classes decompose over their support. In order to make this precise, we have to be careful about what we mean by decomposing spaces equivariantly.

\begin{terminology}\label{term:clopen-components} Let $i \colon Z \hookto X$ be a closed $G$-embedding. As a topological subspace, we may decompose $Z$ non-equivariantly into its clopen components: $Z = \amalg_i W_i$. As $G$ acts via homeomorphisms, we see that the $G$-orbit of any component is both closed and open as well. Thus we may decompose $Z$ as $Z = \amalg G\cdot W_i$, and we refer to the orbits $G\cdot W_i$ as the \textit{equivariant clopen components} of $Z$ in $X$.
\end{terminology}

By collapsing a double mapping cylinder $C_X(X,X-Z)$ down along the time axis, we obtain a ``fried egg'' space as in \autoref{fig:fried-egg}. When $Z$ is decomposed into its equivariant clopen components, we see that the double mapping cylinder decomposes as a wedge sum over the base copy of $X$, as pictured in \autoref{fig:bubble-wrap}.

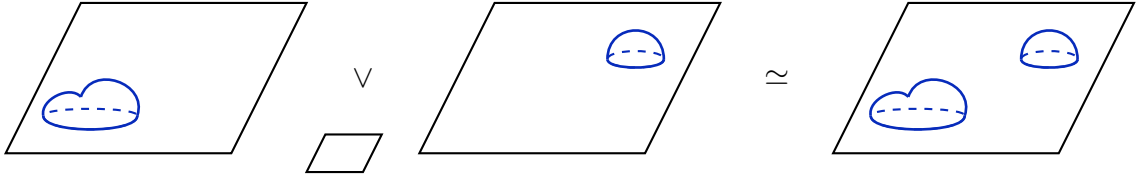
\begin{figure}[H]
  \centering
    \begin{tikzpicture}
	\begin{pgfonlayer}{nodelayer}
		\node [style=none] (0) at (1.25, -1) {};
		\node [style=none] (1) at (2.25, 1) {};
		\node [style=none] (2) at (5.25, 1) {};
		\node [style=none] (3) at (4.25, -1) {};
		\node [style=none] (10) at (0.5, 0) {\(\vee\)};
		\node [style=none] (15) at (6, 0) {\(\simeq\)};
		\node [style=none] (39) at (3.75, 0.25) {};
		\node [style=none] (40) at (4.5, 0.25) {};
		\node [style=none] (41) at (6.75, -1) {};
		\node [style=none] (42) at (7.75, 1) {};
		\node [style=none] (43) at (10.75, 1) {};
		\node [style=none] (44) at (9.75, -1) {};
		\node [style=none] (45) at (7.25, -0.5) {};
		\node [style=none] (46) at (7.75, -0.25) {};
		\node [style=none] (47) at (8.5, -0.5) {};
		\node [style=none] (48) at (9.25, 0.25) {};
		\node [style=none] (49) at (10, 0.25) {};
		\node [style=none] (50) at (-4.25, -1) {};
		\node [style=none] (51) at (-3.25, 1) {};
		\node [style=none] (52) at (-0.25, 1) {};
		\node [style=none] (53) at (-1.25, -1) {};
		\node [style=none] (54) at (-3.75, -0.5) {};
		\node [style=none] (55) at (-3.25, -0.25) {};
		\node [style=none] (56) at (-2.5, -0.5) {};
		\node [style=none] (57) at (-0.25, -1.25) {};
		\node [style=none] (58) at (0, -0.75) {};
		\node [style=none] (59) at (0.75, -0.75) {};
		\node [style=none] (60) at (0.5, -1.25) {};
		\node [style=none] (61) at (3.75, 0.25) {};
		\node [style=none] (62) at (4.5, 0.25) {};
		\node [style=none] (63) at (9.25, 0.25) {};
		\node [style=none] (64) at (10, 0.25) {};
		\node [style=none] (65) at (7.25, -0.5) {};
		\node [style=none] (66) at (8.5, -0.5) {};
		\node [style=none] (67) at (-3.75, -0.5) {};
		\node [style=none] (68) at (-2.5, -0.5) {};
	\end{pgfonlayer}
	\begin{pgfonlayer}{edgelayer}
		\draw [style=black 0.8] (2.center)
			 to (3.center)
			 to (0.center)
			 to (1.center)
			 to cycle;
		\draw [style=xdb line, bend left=90, looseness=0.50] (40.center) to (39.center);
		\draw [style=xdb line, bend left=90, looseness=1.75] (39.center) to (40.center);
		\draw [style=black 0.8] (43.center)
			 to (44.center)
			 to (41.center)
			 to (42.center)
			 to cycle;
		\draw [style=xdb line, bend left=75] (45.center) to (46.center);
		\draw [style=xdb line, bend left=90, looseness=1.50] (46.center) to (47.center);
		\draw [style=xdb line, bend left=90, looseness=0.50] (47.center) to (45.center);
		\draw [style=xdb line, bend left=90, looseness=0.50] (49.center) to (48.center);
		\draw [style=xdb line, bend left=90, looseness=1.75] (48.center) to (49.center);
		\draw [style=black 0.8] (52.center)
			 to (53.center)
			 to (50.center)
			 to (51.center)
			 to cycle;
		\draw [style=xdb line, bend left=75] (54.center) to (55.center);
		\draw [style=xdb line, bend left=90, looseness=1.50] (55.center) to (56.center);
		\draw [style=xdb line, bend left=90, looseness=0.50] (56.center) to (54.center);
		\draw [style=black 0.8] (59.center)
			 to (60.center)
			 to (57.center)
			 to (58.center)
			 to cycle;
		\draw [style=xdb dashed] (62.center)
			 to [bend left=90, looseness=0.50] (61.center)
			 to [bend left=90, looseness=0.50] cycle;
		\draw [style=xdb dashed] (64.center)
			 to [bend left=90, looseness=0.50] (63.center)
			 to [bend left=90, looseness=0.50] cycle;
		\draw [style=xdb dashed] (66.center)
			 to [bend left=90, looseness=0.50] (65.center)
			 to [bend left=90, looseness=0.25] cycle;
		\draw [style=xdb dashed] (68.center)
			 to [bend left=90, looseness=0.50] (67.center)
			 to [bend left=90, looseness=0.25] cycle;
	\end{pgfonlayer}
\end{tikzpicture}
  \caption{If $Z = Z_1 \amalg Z_2$, then we have that $C_X(X,X-Z_1) \vee_X C_X(X,X-Z_2) \simeq C_X(X,X-Z)$.}
  \label{fig:bubble-wrap}
\end{figure}

\begin{proposition}\label{prop:decomp} Take a closed $G$-embedding $i \colon Z\hookto X$, let $Z = \amalg_n Z_n$ be the decomposition of $Z$ into its equivariant clopen components, and denote by $i_n\colon Z_n \hookto X$ the composite inclusion for each $n$. Then there is a weak equivalence 
\begin{align*}
    i_\sharp \Th_Z(Ni) \simeq \vee_n (i_n)_! \mathbb{S}_{Z_n}.
\end{align*}
\end{proposition}
\begin{proof} Via \autoref{prop:kw2}, there is a weak equivalence in $\Ret{G}{X}$ of the form $i_! S^0_Z \simeq C_X(X,X-Z)$, and we may collapse the double mapping cylinder along the time axis as in \autoref{fig:fried-egg}. From there it is clear to see that it can be decomposed as a wedge sum along the equivariant clopen components. The stable version of this statement follows from observing that taking suspension spectra commutes with wedges.
\end{proof}

\begin{example}\label{ex:inclusion-of-orbit} For $G$ a finite group, if $i\colon G/H \to M$ is the closed inclusion of an orbit into a smooth manifold $M$, then there are weak equivalences in $\GOS{M}$ of the form:
\begin{equation}\label{eqn:inclusion-of-orbit-cleaner-version}
\begin{aligned}
    i_! \mathbb{S}_{G/H} \simeq \Sigma^\infty i_\sharp(\pi_{G/H}^\ast \Th(T_x M)) \simeq \Sigma^\infty i_\sharp \left( \Th_{G/H} \left( \left. TM \right|_{ G/H } \right)\right),
\end{aligned}
\end{equation}
where $x$ is any point in the orbit $G/H$.
\end{example}
\begin{proof} By collapsing the double mapping cylinder down around the points in the orbit, we obtain the Thom spaces of the associated tangent spheres at each point in the orbit, glued along $G/H$ to $M$. Note however that for a chosen point $x$ in the orbit, its tangent space inherits an $H$-action,. Thus each Thom space is naturally an $H$-representation sphere $\Th(T_x M)$. The residual $G$-action comes from permuting the representation spheres around between points in the orbit to get $(G/H) \times \Th(T_x M)$. Finally, in order to obtain the collapse of the double mapping cylinders, we glue to $M$ along the orbit $G/H$. This gives an equivalence
\begin{align*}
    C_M(M,M-G/H) \simeq \left( (G/H) \times \Th(T_x M) \right) \cup_{G/H} M,
\end{align*}
This yields the first equivalence in \autoref{eqn:inclusion-of-orbit-cleaner-version}. If $G$ is further assumed to be finite, the tangent space of $G/H$ is trivial, hence the normal bundle $Ni$ agrees with the tangent space $\left. TM \right|_{ G/H }$. In particular we see that
\begin{equation}\label{eqn:thom-space-orbit-tangent}
\begin{aligned}
\Th_{G/H} \left( \left. TM \right|_{ G/H } \right) \simeq \pi_{G/H}^\ast \Th(T_x M).
\end{aligned}\qedhere
\end{equation}
\end{proof}

\begin{corollary}\label{cor:decomposition-cohomology-compact-supports} Cohomology with compact supports decomposes over its support, in the sense that there is a group isomorphism
\begin{align*}
    A_Z^{i^\ast \xi}(X) \cong \bigoplus_n A_{Z_n}^{i_n^\ast\xi}(X).
\end{align*}
\end{corollary}
\begin{proof} We see that \autoref{prop:decomp} induces an isomorphism
\begin{equation*}
\begin{aligned}
     A_Z^{i^\ast \xi} = \left[ i_! \spherespec_Z, \Sigma_X^{\xi} \pi_X^\ast A \right]_X \cong \bigoplus_{n} \left[ (i_n)_! \spherespec_{Z_n}, \Sigma_X^\xi \pi_X^\ast A \right]_X = \bigoplus_n A_{Z_n}^{i_n^\ast \xi}(X).
\end{aligned}\qedhere
\end{equation*}
\end{proof}

\subsection{Integration}

We can push cohomology classes forward along smooth proper families or closed immersions. This comes at the cost of ``untwisting'' by a cotangent complex.
\begin{proposition}\label{prop:pushforward} Let $f\colon X \to Y$ be a smooth $G$-map between smooth compact $G$-manifolds, and suppose that it is either a closed immersion or a smooth proper family. Then for any $\xi\in \Perf(\KO_G(Y))$, the Pontryagin--Thom transfer induces a pushforward
\begin{align*}
    f_\ast \colon A^{\mathcal{L}_f + f^\ast \xi}(X) \to A^{\xi}(Y).
\end{align*}
\end{proposition}
% \begin{proof} We see that $A^{\mathcal{L}_f + f^\ast \xi}(X) = \left[ \spherespec_X, f^! \Sigma_Y^{\xi} \pi_Y^\ast A \right]_X$. Invoking that $f$ is smoothable proper, precomposition with the Pontryagin--Thom transfer $\spherespec_Y \to f_! \spherespec_X$ induces the desired pushforward.
% \end{proof}

\begin{example}\label{ex:pushforward-to-point} Let $M$ be a smooth proper $G$-manifold. Then there is a pushforward
\begin{align*}
    \left( \pi_M \right)_\ast \colon A^{TM}(M) \to A^0(\ast) = \pi_0^G A.
\end{align*}
\end{example}

\begin{proposition}\label{prop:pushforward-or-forget-and-pushforward} Let $Z \subseteq M$ be a closed subspace. Then the following diagram commutes
\[ \begin{tikzcd}
    A_Z^{TM}(M)\dar["\cong" left]\rar["\text{forget}"] & A^{TM}(M)\dar["(\pi_M)_\ast" right]\\
    A^{TZ}(Z)\rar["(\pi_Z)_\ast" below] & A^0(\ast),
\end{tikzcd} \]
where the left vertical isomorphism is the canonical one from \autoref{prop:twisted-tangent-bundle}.
\end{proposition}
\begin{proof} Observe that in the top left we can rewrite
\begin{align*}
    A_Z^{TM}(M) &= \left[ i_! \spherespec_Z, \pi_M^! A \right]_X \cong \left[ \left( \pi_M \right)_! i_! \spherespec_Z, A \right]_X.
\end{align*}
The forgetful map is induced by the Pontryagin--Thom transfer $\spherespec_M \to i_! \spherespec_Z$ as in \autoref{def:forget-supports}, while the pushforward on $M$ is precomposition with the unit $\spherespec \to \left( \pi_M \right)_! \spherespec_M$. The pushforward from $Z$ comes from recognizing that $\left( \pi_M \right)_! i_! = \left( \pi_Z \right)_!$ via \autoref{prop:shrieks-compose}, and using the transfer $\spherespec_M \to \left( \pi_Z \right)_! \spherespec_Z$. The fact that the Pontryagin--Thom transfers along $i$ and $\pi_M$ compose to the unit map along $\pi_Z$ is \autoref{lem:PT-compatibility}.
\end{proof}

\begin{proposition}\label{prop:pushforward-or-decompose-and-pushforward} Let $Z = \amalg_n Z_n$ be a decomposition into its equivariant clopen components, following the notation in \autoref{prop:decomp}. Then the following diagram commutes:
\[ \begin{tikzcd}
    \bigoplus_n A^{TZ_n}(Z_n)\ar[dr,"\oplus_n (\pi_{Z_n})_\ast" above right]\dar["\cong" left]\\
    A^{TZ}(Z)\rar["(\pi_Z)_\ast" below] & A^0(\ast),
\end{tikzcd} \]
where the left vertical map is the decomposition isomorphism in \autoref{cor:decomposition-cohomology-compact-supports}.
\end{proposition}
\begin{proof} We remark that a cohomology class on $A^{TZ_n}(Z_n)$ can be understood by pushing forward directly, or by forgetting support and then pushing forward via \autoref{prop:pushforward-or-forget-and-pushforward}. That is, for any $n$, the diagram commutes:
\[ \begin{tikzcd}
    A^{TZ_n}(Z_n)\ar[dr,"(\pi_{Z_n})_\ast" above right]\dar\\
    A^{TZ}(Z)\rar["(\pi_Z)_\ast" below] & A^0(\ast).
\end{tikzcd} \]
Applying \autoref{cor:decomposition-cohomology-compact-supports}, we see that when we sum over $n$, the left vertical map becomes an isomorphism.
\end{proof}

\subsection{Abstract orientation data}

As we have seen in \autoref{prop:pushforward-or-forget-and-pushforward} and \autoref{prop:pushforward-or-decompose-and-pushforward}, given a cohomology class in $A_Z^{TM}(M)$, we can study it in two ways --- by forgetting its support and pushing it forward, or by decomposing it over its support and pushing each of the individual contributions forward then summing. We have indicated that this is an equality in $\pi_0^G A$.

We will be interested in the more general situation where we are twisting by a bundle $E \to M$ which is not the tangent bundle. To study this, we need to find a way to relate the $E$-twisted cohomology $A^E(-)$ with the cohomology twisted by the tangent bundle $A^{TM}(-)$.

\begin{definition}\label{def:rel-o} We say that a rank $n$ bundle $E \to M$ over a $G$-manifold of dimension $n$ is \textit{relatively $A$-oriented} if there is an isomorphism
\begin{equation}\label{eqn:rel-o}
\begin{aligned}
    \rho \colon \Sigma_M^E \pi_M^\ast A \simeq \Sigma_M^{TM} \pi_M^\ast A.
\end{aligned}
\end{equation}
Such a choice of isomorphism we call a \textit{relative orientation}.
\end{definition}

In \autoref{subsec:cx-orientations}, we will see that complex oriented cohomology theories enjoy a canonical choice of relative orientations, coming from the Thom isomorphism. More generally, once we have a bundle which is relatively oriented in a ring spectrum $A$, we can \textit{push forward} cohomology classes. That is, if $E \to M$ is a rank $n$ complex bundle over a $G$-manifold of dimension $n$, then we can push forward a class in $A^E(M)$ using a relative orientation $\rho$:
\[ \begin{tikzcd}
    A^E(M)\rar["\rho" above]\ar[dr,bend right=10,dashed] & A^{TM}(M)\dar["(\pi_M)_\ast" right] \\
     & A^0(\ast).
\end{tikzcd} \]
Note that if $i \colon Z \hookto M$ is the inclusion of a closed subspace, by applying $i^\ast$ to \autoref{eqn:rel-o}, we obtain an isomorphism of the restricted vector bundle with the tangent bundle on $V$:
\begin{align*}
    \left. \rho \right|_{ Z } \colon \Sigma_Z^{\left. E \right|_{ Z }} \pi_Z^\ast A \simeq \Sigma_Z^{TZ} \pi_Z^\ast A.
\end{align*}
In other words, a relative orientation for $E\to X$ in $A$ descends to compactly supported cohomology groups.

\begin{proposition}\label{prop:compute-locally-in-rel-o} Let $E \to M$ be a rank $n$ complex $G$-vector bundle over a smooth $n$-dimensional $G$-manifold, equipped with a relative $A$-orientation $\rho$. Suppose that $\sigma \colon M \to E$ is a section with zero locus $Z=Z(\sigma)$, which decomposes into clopen components $Z(\sigma) = \amalg_n Z_n$. Then the diagram commutes:
\[ \begin{tikzcd}
    A_Z^E(M)\dar["\cong" left]\rar["\text{forget}"]  & A^E(M)\rar["\rho"]  & A^{TM}(M)\dar["(\pi_M)_\ast" right]\\
    \bigoplus_n A_{Z_n}^{\left. E \right|_{ Z_n }}(M)\rar["\oplus_n \left. \rho \right|_{ Z_n }" below] & \bigoplus_{n} A^{TZ_n}(Z_n)\rar["\oplus_n (\pi_{Z_n})_\ast" below] & A^0(\ast).
\end{tikzcd} \]
\end{proposition}
\begin{proof} This follows directly from \autoref{prop:pushforward-or-forget-and-pushforward} and \autoref{prop:pushforward-or-decompose-and-pushforward}.
\end{proof}
Thus in the presence of a relative orientation, cohomology classes in $A_Z^E(M)$ can be studied by forgetting support and pushing forward, or decomposing, pushing foward, and then summing.

\section{Equivariant conservation of number}\label{sec:conservation}
Here we define refined Euler classes associated to sections of complex vector bundles, valued in an equivariant cohomology theory. \autoref{prop:compute-locally-in-rel-o} indicates that these can be computed as a sum over the local contributions of each of the components of the zero locus of the section. When the cohomology theory $A$ is \textit{complex oriented}, and the zeros are simple and isolated, we demonstrate a tractable formula for the local indices. This gives us an equality in $\pi_0^G A$ which is independent of the choice of section.
\subsection{Refined Euler classes}

We fix some notation for this section.

\begin{setup}\label{set:euler-class-setup} Let $G$ be a finite group, and let $E \to M$ be a $G$-equivariant vector bundle of dimension $n$ over a smooth compact $n$-manifold (real or complex). Let $Z = Z(\sigma)$ be the zero locus, with inclusion map $i \colon Z \hookto M$. Let $A \in \mathrm{CAlg}(\Sp_G)$ be a genuine $G$-ring spectrum.
\end{setup}

The section $\sigma$ induces a map of pairs $(M,M-Z) \to (E,E-M)$, where $M \subseteq E$ is the zero section. This induces an equivariant map of double mapping cylinders $C_M(M,M-Z) \to C_E(E,E-M)$, which can be regarded as a map
\begin{equation}\label{eqn:double-mapping-cylinder-pairs}
\begin{aligned}
    i_! \mathbb{S}_Z \to \Sigma^E_M \mathbb{S}_M.
\end{aligned}
\end{equation}
\begin{definition}\label{def:refined-euler-class} In the language of \autoref{set:euler-class-setup}, denote by $e(E,\sigma,Z) \in A^E_{Z}(M)$ the \textit{refined Euler class}, defined to be the composite
\begin{align*}
    i_! \mathbb{S}_Z \to \Sigma^E_M \mathbb{S}_M \xto{\Sigma_M^E 1} \Sigma_M^E \pi_M^\ast A,
\end{align*}
where the first map is \autoref{eqn:double-mapping-cylinder-pairs} and the second is the unit on $A$.
\end{definition}

Decomposing $Z$ into its equivariant clopen components $Z = \amalg Z_n$ (as in \autoref{term:clopen-components}), we can invoke \autoref{cor:decomposition-cohomology-compact-supports} to decompose the Euler class over its support:
\begin{align*}
    A_Z^E(M) &\xto{\sim} \bigoplus_n A_{Z_n}^E(M) \\
    e(E,\sigma,Z) &\mapsto \oplus_n e(E,\sigma, Z_n).
\end{align*}

\begin{definition}\label{def:local-index} When $E$ is equipped with a relative $A$-orientation, the image of $e(E,\sigma,Z_n)$ under pushforward is referred to as the \textit{local index}, denoted by $\ind_{Z_n}(\sigma)$:
\begin{align*}
    A^{E}_{Z_n}(M) \cong A^{TZ_n}(Z_n) \to &A^0(\ast) \\
    e(E,\sigma,Z_n) \longmapsto &\ind_{Z_n}(\sigma).
\end{align*}
We refer to the image of the (un)refined Euler class under pushforward as the \textit{Euler number}, and denote it by $n(E,\sigma)$:
\begin{align*}
    A_Z^{E} \xto{\text{forget}} A^{TM}(M) \to A^0(\ast)=\pi_0^GA \\
    e(E,\sigma,Z) \longmapsto n(E,\sigma).
\end{align*}
We are suppressing the notation for the relative orientation, but the reader should remark that these quantities depend upon the choice of relative orientation.
\end{definition}
With this terminology in hand, we can state the following lemma.
\begin{lemma}\label{lem:conservation}
In \autoref{set:euler-class-setup}, if $A$ is equipped with a relative $A$-orientation, then we obtain an equality in $\pi_0^G A$ of the form
\begin{align*}
    n(E,\sigma) = \sum_n \ind_{Z_n}(\sigma).
\end{align*}
Moreover, the value $n(E,\sigma)$ is independent of the choice of section $\sigma$, and only depends upon the relative orientation.
\end{lemma}
\begin{proof} We obtain the desired equality by following the Euler class $e(E,\sigma,Z)$ in the commutative diagram of \autoref{prop:compute-locally-in-rel-o}.
Thus we can compute the Euler number by decomposing it over its support, and summing over the local indices. To see that this is independent of $\sigma$, we remark that $n(E,\sigma)$ was defined up to the homotopy class of $\sigma$. Since every section can be $G$-equivariantly homotoped to the zero section, we observe that $n(E,\sigma)$ is independent of $\sigma$.
\end{proof}

\begin{terminology}
We say that $\sigma$ has an \textit{isolated zero} at a point $x\in M$ if $\{x\}$ is both closed and open in $Z(\sigma) \subseteq M$. We say $\sigma$ has a \textit{simple zero} at $x\in M$ if the zero is simple in the classical sense --- meaning the Jacobian determinant of $\sigma$ at $x$ is non-vanishing. Note that the action of $G$ preserves the properties of being isolated and simple.
\end{terminology}

\begin{proposition}\label{prop:local-euler-class-at-orbit} In \autoref{set:euler-class-setup}, if $G/H \subseteq Z$ is a clopen component, and $x\in G/H$ is a simple isolated zero of $\sigma$, then there is a natural equivalence
\begin{align*}
    A_{G/H}^{E}(G/H) &\cong \left[ \pi_{G/H}^\ast \Th(T_xM), \Sigma_{G/H}^{\left. E \right|_{ G/H }} \pi_{G/H}^\ast A \right],
\end{align*}
under which the Euler class $e(M,\sigma,G/H)$ corresponds to the composite of the intrinsic derivative $d_x \sigma$ and the unit map on $A$:
\begin{align*}
    (G/H) \times \Th(T_x M) \to (G/H) \times\Th(E_x) \to (G/H) \times (\Th(E_x) \smashprod A).
\end{align*}
\end{proposition}
\begin{proof} Via \autoref{ex:inclusion-of-orbit}, the Euler class $e (M,\sigma, G/H)$ can be considered as a composite of the form
\begin{align*}
    i_\sharp j_\sharp \pi_{G/H}^\ast \Th(T_x M) \to \Sigma_M^E \spherespec_M,
\end{align*}
where $x\in G/H$ is any point in the orbit. By adjunction this is the same as \begin{align*}
    \pi_{G/H}^\ast \Th(T_x M) \to j^\ast i^\ast \Sigma_M^E \spherespec_M = \Sigma_{G/H}^{\left. E \right|_{ G/H }} \spherespec_{G/H}.
\end{align*}
If $x$ is assumed to be simple, then \autoref{eqn:thom-space-orbit-tangent} tells us this is precisely the map
\begin{align*}
    (G/H) \times \Th(T_x M) \xto{G/H \times d_x\sigma} (G/H) \times \Th(E_x),
\end{align*}
where $d_x\sigma$ is the intrinsic derivative of $\sigma$ at the point $x$.
\end{proof}

\begin{remark}\label{rmk:need-orientation-data}  Here is where orientation data is needed. We have an induced map between $H$-representation spheres of the same dimension, but this does not canonically give a class in the $H$-Burnside ring. The fact is while $S^{T_x M - E_x}$ is isomorphic to $S^0$ when $x$ is a simple zero, one must \textit{fix an isomorphism}, and there is no canonical way to do this. To circumvent this issue, we look at the map that the intrinsic derivative $S^{T_x M} \xto{d_x \sigma} S^{E_x}$ induces on a cohomology theory $A$, where $A$ comes equipped with some canonical orientation data. In particular for such a ring spectrum $A$, we get a composite:
\begin{align*}
    S^{T_x M} \xto{d_x \sigma \smashprod u} S^{E_x} \smashprod A \xto{\text{orientation data}} S^{T_x M} \smashprod A.
\end{align*}
This gives us a well-defined class in $\pi_0^H A$, and the associated local index is obtained by transferring this up to $G$ along the transfer available to us in the zeroth homotopy Mackey functor $\underline{\pi}_0 A$.\footnote{For the reader who may be unfamiliar with the language of Mackey functors, we refer them to the excellent introductory paper \cite{Webb}.} In practice we will be concerned with complex oriented equivariant ring spectra, where this ``orientation data'' is the data of a Thom class arising from a universal one.
\end{remark}

\subsection{Complex orientations in the equivariant setting}\label{subsec:cx-orientations}

\begin{definition} Let $\mathcal{U}$ denote a direct sum of infinitely many copies of each irreducible complex representation of $G$, and denote by
\begin{align*}
    \BU_G(n) := \Gr(\C^n, \mathcal{U}),
\end{align*}
the moduli space of $n$-dimensional subspaces of $\mathcal{U}$. Since $G$ acts naturally on $\mathcal{U}$, it acts on $\BU_G(n)$ as well, and $\BU_G(n)$ comes equipped with a tautological bundle $\gamma_n \colon \EU_G(n) \to \BU_G(n)$ which is easily seen to be equivariant.
\end{definition}

Following tom Dieck \cite{tD}, we may assemble the Thom spaces of the bundles $\Th(\gamma_n)$ into a genuine $G$-spectrum by setting the $V$th space equal to $\Th(\gamma_{|V|})$, and then spectrifying (see \cite{Sinha} for a lucid overview). This definition yields \textit{equivariant homotopical bordism}, which we denote by $\MU_G$.

Combining the work of tom Dieck \cite{tD} and Okonek \cite[\S1]{Okonek-CF}, we make the following definition.

\begin{definition}\label{def:complex-oriented} Let $G$ be a compact Lie group, and let $A$ be a multiplicative $\RO(G)$-graded cohomology theory. Define a \textit{complex orientation} on $A$ to be a choice, for every complex vector bundle $p \colon  E \to X$ of complex rank $k$, of \textit{Thom classes} $\tau(p) \in \til{A}^{2k}(\Th(E))$ subject to the following conditions:
\begin{enumerate}

    \item[(0)] (Thom isomorphisms) Cupping with the Thom class $\tau(p)$ induces a \textit{Thom isomorphism}:
    \begin{align*}
        A^\ast(-) \xto{\tau(p) \cup -} \til{A}^{\ast+2k}(- \smashprod \Th(E)).
    \end{align*}
    
    \item (Naturality) The pullback of a Thom class is the Thom class of the pullback bundle.

    \item (Multiplicativity) The Thom class of a product bundle is the product of the Thom classes of the respective bundles in the product.

    \item (Unitality) For any rank $n$ representation $V$, viewed as an equivariant bundle over a point, its Thom class $\tau(V)$ is the image of $1\in A^0(\ast)$ under the Thom isomorphism.
\end{enumerate}
\end{definition}

\begin{remark} The unitality condition dates back to tom Dieck and Okonek. Depending on preference, one might drop this condition and obtain a more general notion of complex orientations, in which we are allowed to rescale all our Thom classes by a unit in $\pi_0^G A$, for instance.
\end{remark}

\begin{note} In classical homotopy theory, the data of a complex orientation compresses to a single Thom class for the universal line bundle. This reduction relies on having a strong handle on the cell structure of the classifying space $\BU(n)$ for complex line bundles. In the equivariant setting, a filtration of $\BU_G(n)$ by equivariant Schubert cells is made complicated by the possible existence of irreducible $G$-representations of higher dimension. When $G$ is abelian, all irreducible representations are one-dimensional, and the cell structure on $\BU_G(n)$ is better understood, so such a compression is possible, and described in \cite{CGK-univ}. This is the primary reason that our understanding of the connection between equivariant complex orientations and equivariant formal group laws is mostly limited these days to the case of abelian groups (see for instance \cite{CGK-FGL,Hausmann}).
\end{note}

\begin{remark} For $V$ a complex rank $n$ representation, and $A$ a complex oriented cohomology theory, $\tau(V)$ is a map $S^V \to \Sigma^n A$. We observe that the following composite is the Thom isomorphism, which we will also denote by $\tau(V)$:
\begin{align*}
    A \smashprod S^V \xto{1\smashprod \tau(V)} A \smashprod \Sigma^n A \xto{\mu} \Sigma^n A,
\end{align*}
where $\mu$ denotes the multiplication on the ring spectrum. In other words, $\Sigma^V A \simeq \Sigma^{|V|} A$. This is the notion of $\GL$-orientation one encounters e.g. in \cite[4.13]{BW3}.
\end{remark}

\begin{example}\cite{Okonek-CF} For any compact Lie group $G$, homotopical bordism $\MU_G$ admits a complex orientation.
\end{example}

\begin{theorem} Given a compact Lie group $G$, a unital ring map $\MU_G \to A$ endows $A$ with a complex orientation. If $G$ is furthermore assumed to be abelian, then this is an equivalent definition of complex orientation \cite[Lemma~1.6]{Okonek-CF}, \cite[Theorem~1.2]{CGK-univ}.
\end{theorem}

\begin{example}\cite{Okonek-CF,Costenoble} For any compact Lie group $G$, complex equivariant $K$-theory $\KU_G$ receives a ring map $\MU_G \to \KU_G$ and is therefore complex oriented.
\end{example}

\begin{counterexample} Eilenberg--Maclane spectra of Mackey functors $H\underline{M}$ are generally \textit{not} complex oriented, in stark contrast to the non-equivariant setting. By pulling Thom classes back along the zero section, we obtain Euler classes in cohomology. If $V$ is a $G$-representation of dimension $n$, then pulling back the Thom class of its representation sphere along the zero section $S^0 \to S^V$ yields a class in $\pi_{-n}^G \MU_G$. This class is generally nonzero, indicating that $\MU_G$ is non-connective. All Eilenberg--MacLane spectra are integrally connective, hence in order to create a ring map $\MU_G \to H\underline{M}$, we would have to send Euler classes to zero, which destroys any possibility of the map preserving information about orientation.
\end{counterexample}

\begin{remark}\label{rmk:thom-iso}  If $A$ is a complex oriented cohomology theory, and $V$ and $W$ are complex $G$-representations of the same dimension $n$, by unitality, we obtain isomorphisms
\begin{align*}
    A^n(S^V) \cong A^0(\ast) \cong A^n(S^W).
\end{align*}
By unitality of the complex orientation, this is an isomorphism of free $\pi_0^G A$-modules of rank one, sending $\tau(V)\mapsto \tau(W)$. We may also refer to this a \textit{Thom isomorphism} by abuse of terminology.
\end{remark}

We record an important property enjoyed by complex oriented ring spectra in the equivariant setting. Informally, the following propositions state that any isomorphism of $G$-representations also represents the Thom isomorphism obtained by passing between the two representations in any complex oriented cohomology theory.

\begin{proposition}\label{prop:rep-isos-thom-classes}  Let $A$ be a complex oriented ring spectrum, and let $f \colon V \xto{\sim} W$ be any isomorphism of complex $G$-representations of dimension $n$. Then the isomorphism $f^\ast \colon A^n(S^W) \cong A^n(S^V)$ has as its inverse the Thom isomorphism of \autoref{rmk:thom-iso}.
\end{proposition}
\begin{proof} Since this is an isomorphism of free $\pi_0^G A$-modules of rank one, we have to check where the generator is sent, and we observe that $f^\ast \tau(W) = \tau(V)$ by naturality of the complex orientation.
\end{proof}

\begin{remark} At no point in \autoref{prop:rep-isos-thom-classes} did we use any specific properties of the choice of isomorphism $f$. This is unsuprising, due to the fact that all isomorphisms of complex representations $V_1 \xto{\sim} V_2$ are homotopic \cite[1.1]{tD}, thus there is a single homotopy class $[S^{V_1}, S^{V_2}]$ corresponding to isomorphisms of representations. The argument above indicates roughly that after smashing with $A$, this homotopy class aligns with that produced by the Thom isomorphism.
\end{remark}

We can now revisit our discussion of local indices from \autoref{rmk:need-orientation-data}.

\subsection{Local indices and conservation of number}

\begin{lemma}\label{lem:local-index} Let $A$ be any complex oriented ring spectrum in $\Sp_G$, let $E \to M$ be an equivariant complex vector bundle of rank $n$ over a compact smooth $G$-manifold of dimension $n$, and let $\sigma \colon  M \to E$ be a section with an isolated simple zero at $x\in M$. Then the local index, as defined in \autoref{def:local-index}, is
\begin{align*}
    \ind_{G\cdot x} \sigma = \Tr_{G_x}^G(1).
\end{align*}
\end{lemma}
\begin{proof} We must argue that the composite
\begin{align*}
    S^{T_x M} \smashprod A \xto{d_x \sigma \smashprod A} S^{E_x} \smashprod A \xto{\tau} S^{T_x M} \smashprod A
\end{align*}
is equal to $1 \in \pi_0^G A$, where $\tau$ is arising from the Thom classes provided by the equivariant complex orientation on $A$. As $x$ is an isolated simple zero, the intrinsic derivative is an injective map of $G$-representations of the same finite dimension, and hence is an isomorphism $d_x \sigma \colon T_x M \to E_x$. Thus we find ourselves under the conditions of \autoref{prop:rep-isos-thom-classes}, from which the result follows.
\end{proof}

To wrap up this section, we explore a payoff of the formalism developed above, which will serve as our primary computational tool. Namely, we can develop a theory of conservation of number taking value in $\pi_0^G A$ for any complex oriented equivariant cohomology theory $A$.

By \autoref{lem:local-index}, the local index at an isolated simple orbit $G\cdot x$ is the trace $\Tr_{G_x}^G(1)$ from the isotropy group of $x$ to the entire group $G$, where this transfer is taking place at the level of the zeroth homotopy Mackey functor. The following key lemma should be thought of as an equivariant analogue of the Poincar\'e--Hopf theorem, with cohomology classes valued in complex oriented $G$-ring spectra.

\begin{lemma}\label{lem:cons-number-cx-oriented-ring-spectra} \textit{(Equivariant conservation of number)} Let $E \to M$ be an equivariant complex rank $n$ bundle over a smooth $G$-manifold of dimension $n$, and let $\sigma \colon  M \to E$ be any section whose zeros are isolated and simple. Let $A\in\Sp_G$ be any complex oriented ring spectrum. Then there is an equality in $\pi_0^G A$:
\begin{align*}
    n(E,\sigma) = \sum_{G\cdot x \subseteq Z(\sigma)} \Tr_{G_x}^G(1),
\end{align*}
where the Euler number $n(E,\sigma)$ is independent of the choice of $\sigma$.
\end{lemma}

\begin{example}\label{ex:KUG-valued-Euler-class} In complex $K$-theory, we have that $\KU_{G_x}(\ast) = R_\mathbb{C}[G_x]$, and the transfer of the trivial representation $1$ is the regular representation of the finite $G$-set $G/G_x$. Thus an Euler number computed as in \autoref{lem:cons-number-cx-oriented-ring-spectra} is given by the permutation representation $\C[Z(\sigma)]$ of the zero locus of a section with isolated simple zeros, and the conservation statement is that $\C[Z(\sigma)] \cong \C[Z(\sigma')]$ is an isomorphism of $G$-representations for any two sections with simple isolated zeros.
\end{example}

Ultimately we want to argue that an answer valued in the \textit{Burnside ring} $A(G)$ is independent of a choice of section. The $\KU_G$-valued Euler class as in \autoref{ex:KUG-valued-Euler-class} is insufficient for this purpose, due to the fact that the map $\pi_0^G \mathbb{S}_G \to \pi_0^G \KU_G$ from the Burnside ring to the representation ring will often fail to be injective. We instead need a complex oriented cohomology theory for which the unit map is an injection on $\pi_0^G$.

We thank William Balderrama for communicating the following argument to us.

\begin{proposition}\label{prop:MUG-detects-nilpotents} Homotopical bordism $\MU_G$ detects nilpotence, in the sense that for any ring spectrum $A$ equipped with a ring map $A \to \MU_G$, the kernel of $\pi_\star^G A \to \pi_\star^G \MU_G$ consists of nilpotent elements (see \cite[3.20]{BGH}).
\end{proposition}
\begin{proof} Taking geometric fixed points commutes with the construction of a mapping telescope, which allows us to conclude that nilpotence can be detected at the level of geometric fixed points \cite[3.17]{BGH}. By \cite[4.10]{Sinha}, $\Phi^H \MU_G$ decomposes as a wedge sum of classical $\MU$ spectra. Finally, we can conclude by applying the classical nilpotence theorem \cite{DHS}.
\end{proof}

We leverage this to prove our main result.

\begin{theorem}\label{thm:cons-number} \textit{(Equivariant conservation of number)}  Let $E \to M$ be an equivariant complex rank $n$ bundle over a smooth $G$-manifold of dimension $n$, and let $\sigma,\sigma' \colon  M \to E$ be any two sections whose zeros are isolated and simple. Then $Z(\sigma)$ and $Z(\sigma')$ are isomorphic as finite $G$-sets. In other words, the $G$-orbits of the zeros are independent of the choice of section.
\end{theorem}
\begin{proof} We know for such a section $\sigma$, we can obtain an Euler class valued in $\pi_0^G \MU_G$ by \autoref{lem:cons-number-cx-oriented-ring-spectra}:
\begin{align*}
    n(E,\sigma) = \sum_{G\cdot x \subseteq Z(\sigma)} \Tr_{G_x}^G(1).
\end{align*}
As $\MU_G$ detects nilpotence by \autoref{prop:MUG-detects-nilpotents}, and the Burnside ring is reduced, we can conclude that $\pi_0^G \mathbb{S}_G \to \pi_0^G \MU_G$ is injective. Remarking that the map $\underline{\pi}_0 \mathbb{S}_G \to \underline{\pi}_0 \MU_G$ is a map of Tambara functors, we can observe that $n(E,\sigma)$ admits a \textit{unique preimage} in $A(G)$, given by $\Tr_{G_x}^G(1)$, where this transfer is of the form $\Tr_{G_x}^G \colon A(G_x) \to A(G)$. This is precisely the $G$-set $Z(\sigma)$.\end{proof}

In the following section we leverage this perspective to compute the equivariant count of 27 lines on a symmetric smooth cubic surface.

\section{The 27 lines on a smooth symmetric cubic surface}\label{sec:27}

In this section we apply our methods to compute the orbits of lines on a smooth symmetric cubic surface. In particular in the presence of symmetry we can state further constraints about the number of lines defined on a real cubic surface.

\subsection{27 lines on a complex symmetric cubic surface}

\begin{definition} We say that a cubic surface $X = V(F)\subset \P^3$ is $S_4$-\textit{symmetric} (or just \textit{symmetric}) if $F(x_0,x_1,x_2,x_3)$ is a symmetric polynomial.
\end{definition}

In particular by letting $S_4$ act on $\CP^3$ by permuting projective coordinates, we have that symmetric cubics are precisely those preserved under this action. The lines on such a cubic surface therefore come equipped with $S_4$-orbits, and we can inquire about the orbit type. By equivariant conservation of number, the answer is independent of the choice of symmetric cubic surface.

\begin{theorem}\label{thm:27-lines} Given any smooth symmetric complex cubic surface, its 27 lines have orbit type
\begin{align*}
    \left[ S_4/C_2^o \right] + \left[ S_4/C_2^e \right]+ \left[ S_4/D_8 \right],
\end{align*}
where $C_2^o$ is a single transposition, and $C_2^e$ is a product of two disjoint transpositions.
\end{theorem}
\begin{proof} We remark that a symmetric complex cubic surface $X$ induces a section of the following $S_4$-equivariant complex vector bundle:
\[ \begin{tikzcd}
    \Sym^3 \mathcal{S}^\ast \rar & \Gr_\mathbb{C}(1,\CP^3),\lar[bend right=10,"\sigma_X" above]
\end{tikzcd} \]
where $\mathcal{S}$ denotes the tautological bundle on the Grassmannian. In particular $\sigma_X(\ell) = 0$ if and only if $\ell \subseteq X$ is a line on the symmetric cubic. Since the 27 lines on $X$ are necessarily distinct (c.f.~\cite[Theorem~5.1]{3264}), the zero locus $Z(\sigma_X)$ consists of 27 points on $\Gr_\mathbb{C}(1,\CP^3)$, each of which is a simple zero of $\sigma_X$.

By \autoref{thm:cons-number}, the $S_4$-orbits will be independent of the choice of symmetric cubic, so it suffices to pick our favorite symmetric cubic and compute the $S_4$-orbits of its lines. Consider the example of the \textit{Fermat cubic}:
\begin{align*}
    F = \left\{ [x_0:x_1:x_2:x_3] \colon x_0^3 + x_1^3 + x_2^3 + x_3^3 = 0 \right\}.
\end{align*}
Fix $\zeta$ to be a primitive sixth root of unity in $\C$, hence we have three distinct cube roots of $-1$, namely $\zeta$, $\zeta^{-1}$, and $-1$. The 27 lines on the Fermat are given by the following equations, where $[w:z]$ varies over $\CP^1$:
\begin{center}
    \begin{tabular}{l l l l}
    \blue{$[w: - w : z : \zeta z]$} & \blue{$[w: - w : z : \zeta^{-1} z]$} & \blue{$[w:\zeta w : z : - z]$} &  \blue{$[w:\zeta^{-1} w : z : - z]$}  \\
    \blue{$[w: z: \zeta w : - z]$} & \blue{$[w: z: \zeta^{-1} w : - z]$} & \blue{$[w: z: - w : \zeta z]$} & \blue{$[w: z: - w : \zeta^{-1} z]$} \\
    \blue{$[w: z: -z : \zeta w]$} &  \blue{$[w: z: -z : \zeta^{-1} w]$} & \blue{$[w: z: \zeta z : - w]$} & \blue{$[w: z : \zeta^{-1} z : - w]$} \\
    \green{$[w:\zeta w : z : \zeta z]$} & \green{$[w:\zeta w : z : \zeta^{-1} z]$} & \green{$[w:\zeta^{-1} w : z : \zeta z]$} & \green{$[w:\zeta^{-1} w : z : \zeta^{-1} z]$} \\
    \green{$[w: z: \zeta w : \zeta z]$} & \green{$[w: z: \zeta w : \zeta^{-1} z]$} & \green{$[w: z: \zeta^{-1} w : \zeta z]$} & \green{$[w: z: \zeta^{-1} w : \zeta^{-1} z]$} \\
    \green{$[w: z: \zeta z : \zeta w]$} & \green{$[w: z : \zeta^{-1} z : \zeta w]$}&  \green{$[w: z: \zeta z : \zeta^{-1} w]$} & \green{$[w: z : \zeta^{-1} z : \zeta^{-1} w]$}  
    \end{tabular}
    \begin{tabular}{l l l}
    \red{$[w: - w : z : - z]$} & \red{$[w: z: - w : - z]$} & \red{$[w: z: -z : - w]$}
    \end{tabular}
\end{center}
Thus the orbits are as follows (colors are chosen so that the orbits match the orbits of the lines on \autoref{fig:clebsch}), where the notation $C_2^o$ (odd) denotes a single transposition and $C_2^e$ (even) denotes a product of two disjoint transpositions.
\begin{center}
    \begin{tabular}{l | l | l | l | l}
    Color & Generating line & Isotropy subgroup & Orbit type & \# of lines\\
    \hline
    \blue{Blue} & \blue{$[w: -w : \zeta z : -z]$} & $\left\langle (1\ 2) \right\rangle$ & $S_4 / C_2^o$  & 12\\
    \green{Green} & \green{$[w: \zeta w: z : \zeta z]$} & $\left\langle (1\ 3)(2\ 4) \right\rangle$ & $S_4/C_2^e$ & 12 \\    
    \red{Red} & \red{$[w:-w:z:-z]$} & $\left\langle (1\ 3)(2\ 4),(1\ 2),(3\ 4) \right\rangle$ & $S_4 / D_8$ & 3
    \end{tabular}
\end{center}

\end{proof}

Given a subgroup $G \le S_4$, it induces a natural action on $\CP^3$, and we can ask about the 27 lines on a $G$-symmetric smooth complex cubic surface. For this result, the following is key.

\begin{proposition}\label{prop:restriction-Euler-number} Let $E \to M$ be a $G$-equivariant rank $n$ complex vector bundle over a smooth compact $G$-manifold of dimension $n$. Let $A$ be a complex oriented $G$-cohomology theory, let $n_G(E)$ denote the Euler number of $E$ in the $A$-cohomology theory, and for a subgroup $H \le G$, let $n_H(E)$ denote the Euler number in the restricted $H$-equivariant $A$-cohomology theory. Then we have that
\begin{align*}
    n_H(E) = \Res_H^G n_G(E).
\end{align*}
\end{proposition}
\begin{proof} This follows directly from the computation of the Euler number along a section with isolated simple zeros being valued in the homotopy Mackey functor $\underline{\pi}_0 A$.
\end{proof}

Thus given any subgroup $G \le S_4$, we can compute the regular representation of the orbits of its 27 lines under the associated $G$-action by restricting the answer for $S_4$.

\begin{notation}\label{nota:subgps-s4} 
We denote by $C_2^o := \left\langle  (1\ 2)\right\rangle$ and $C_2^e := \left\langle (1\ 2)(3\ 4) \right\rangle$ the odd and even conjugacy classes of cyclic subgroups of order two in $S_4$. We denote by $K_4^\triangleleft$ the normal Klein 4-subgroup of $S_4$, and $K_4$ the non-normal one. For a Klein four group, we denote by $C_2^L$, $C_2^R$, and $C_2^\Delta$ the left, right, and diagonal cyclic subgroups of order two, respectively.
\end{notation}

\begin{corollary}\label{cor:subgps-s4} For all of the conjugacy classes of subgroups $G \le S_4$, we can compute the $G$-orbits of the 27 lines on a $G$-symmetric smooth complex cubic surface, where the $G$-action on $\CP^3$ is acting on coordinates. These are in Table 1.
\end{corollary}

\begin{proof} Following \autoref{prop:restriction-Euler-number}, the orbits listed are the $G$-orbits of the 27 lines on the Fermat cubic.
\end{proof}

\begin{remark} We point out that the computations in \autoref{cor:subgps-s4} demonstrate the full strength of working with an $\MU_G$-valued Euler class, rather than a $\KU_G$-valued Euler class, for instance. In fact \autoref{thm:27-lines} can be proven with a $\KU_{S_4}$-valued Euler class; we first compute the permutation representation of the 27 lines on a symmetric cubic surface, and then look to argue there is a unique $S_4$-set with this given permutation representation. While the Burnside ring homomorphism $A(S_4) \to R_\C(S_4)$ is not injective, by looking for an honest $S_4$-set (i.e. a point in the upper orthant of $A(S_4)\cong \Z^{11}$), we obtain a system of integral linear equations and inequalities, and we can verify on a computer using polyhedral methods that a unique solution exists. In trying to replicate this argument for all subgroups of $S_4$, we see that these polyhedral techniques fail for both copies of the Klein four-group, as well as the dihedral group $D_8 \le S_4$. Thus in order to obtain \autoref{cor:subgps-s4} the $\MU_G$-valued Euler class is needed.
\end{remark}

\begin{wraptable}{R}{0.5\textwidth}
\tiny
\renewcommand{\arraystretch}{1.3}%
    \begin{tabular}{|c | p{4.5cm}|}
    \hline
    \textbf{Subgroup} $G \le S_4$ & $G$-\textbf{orbits of lines}  \\
    \hline
    $e$ & $27[e/e]$ \\
    $C_2^o$ & $12 [C_2/e] + 3[C_2/C_2]$ \\
    $C_2^e$ & $10[C_2/e] + 7[C_2/C_2]$ \\
    $C_3$ & $9 [C_3/e]$ \\
    $K_4^\triangleleft$ & $[K_4/e] + 4[K_4/C_2^L] + 4[K_4/C_2^R] + 2[K_4/C_2^\Delta] + 3[K_4/K_4]$ \\
    $K_4$ & $4[K_4/e] + [K_4/C_2^L] + [K_4/C_2^R] + 3[K_4/C_2^\Delta] + [K_4/K_4]$ \\
    $C_4$ & $5[C_4/e] + 3[C_4/C_2^o] + [C_4/C_4]$ \\
    $S_3$ & $3[S_3/e] + 3[S_3/C_2^o]$ \\
    $D_8$ & $\left[ D_8/e \right] + 3[D_8 / C_2^e] + [D_8/C_2^o] + [D_8/K_4] + [D_8/D_8]$ \\
    $A_4$ & $\left[ A_4/e \right] + 2 \left[ A_4/C_2^e \right] + \left[ A_4/K_4 \right]$ \\
    $S_4$ & $\left[ S_4/C_2^o \right] + \left[ S_4/C_2^e \right] + \left[ S_4/D_8 \right]$ \\
    \hline
    \end{tabular}
    \vspace{-1em}
\label{table:subgroups-S4}
\caption{$G$-orbits of the 27 lines on a cubic surface, for $G \subseteq S_4$}
\end{wraptable}

\subsection{27 lines on a real symmetric cubic}

Observe that in the proof of \autoref{thm:27-lines}, the three lines in the orbit $[S_4/D_8]$, labeled in red, were in fact defined over the reals. This is true in general.

\begin{proposition}\label{prop:orbit-of-3-lines-is-real} Let $F$ be a real smooth symmetric cubic. Then on its complexification $V(F_\mathbb{C})$, the lines in the orbit $[S_4/D_8]$ are all defined over the reals, and hence form an orbit $[S_4/D_8]$ on $V(F)$.
\end{proposition}
\begin{proof} Since lines defined over $\mathbb{C}$ but not over $\R$ must come in complex conjugate pairs, any such orbit of lines must be of even size. Since $|S_4/D_8| = 3$, all of its lines must in fact be real.
\end{proof}

The study of rationality of lines on a real cubic surface is a classical problem dating back to the mid-1800's.

\begin{theorem} \cite{Schlafli} A real smooth cubic surface can only contain 3, 7, 15, or 27 real lines, and all of these possibilities do in fact occur.
\end{theorem}

\autoref{prop:orbit-of-3-lines-is-real} actually implies more --- by examining the possible fields of definition of the other orbits of 12 lines, we can easily eliminate the possibility of seven real lines on a real symmetric cubic surface. We can do better by using more refined information about the lines in question, namely their topological \textit{type}.

\begin{definition}\label{def:type-of-line} Let $\ell$ be a line on a smooth real cubic surface $X$, and consider the map
\begin{align*}
    \RP^1 \cong \ell &\to \SO(3) \\
    x &\mapsto T_x X.
\end{align*}
This associates to each line $\ell$ on the cubic surface a loop in the frame bundle $\pi_1(\SO(3))=\Z/2 = \left\{ \pm 1 \right\}$. The line $\ell$ is said to be \textit{hyperbolic} if the associated class is $+1 \in \Z/2$, and \textit{elliptic} if its associated class is $-1 \in \Z/2$. We refer to this as the \textit{type} of the line $\ell \subseteq X$.
\end{definition}

\begin{proposition}\label{prop:action-preserves-type} On a real symmetric cubic surface $X$, the $S_4$ action on $\CP^3$ by permuting coordinates preserves the topological type of any line.
\end{proposition}
\begin{proof} Given a line $\ell$ on a real cubic $X$, we have that for any point $p\in \ell$, there is a uniquely determined point $q \in \ell$ so that their tangent spaces are equal: $T_p X = T_q X$. This allows us to define an involution of the line $\ell$, given by sending $p \mapsto q$ for every such pair of points. The topological type of the line is equivalently defined via the discriminant of the fixed locus of this involution \cite{FinashinKharlamov}. Since this involution is defined independent of coordinates, it is invariant under a change of coordinates, and therefore the $S_4$-action does not affect the geometric properties of the involution attached to a line on $X$.
\end{proof}

This indicates that within an $S_4$-orbit, all lines have the same type. A classical result following from work of Segre indicates that the types of lines are constrained.

\begin{theorem} \cite{Segre,Benedetti-Silhol,OkonekTeleman,FinashinKharlamov,KW} Let $X$ be a real smooth cubic surface. Then the following equality holds:
\begin{align*}
    \# \left\{ \text{real hyperbolic lines on } X \right\} - \# \left\{ \text{real elliptic lines on } X \right\} = 3.
\end{align*}
\end{theorem}

Combining this with Schl\"afli's result, we have the following possibilities for real lines on a real smooth cubic surface:

\begin{center}
    \begin{tabular}{c | c c}
    $\substack{\text{Total number} \\ \text{of real lines}}$  & $\substack{\text{Number of} \\ \text{hyperbolic lines}}$ & $\substack{\text{Number of} \\ \text{elliptic lines}}$ \\ 
    \hline
    3 & 3 & 0 \\
    7 & 5 & 2 \\
    15 & 9 & 6 \\
    27 & 15 & 12.
    \end{tabular}
\end{center}

\begin{theorem}\label{thm:27-real} A real smooth symmetric cubic surface can only contain 3 or 27 real lines, and both of these possibilities do occur.
\end{theorem}
\begin{proof} By the argument following \autoref{prop:orbit-of-3-lines-is-real}, we have that the possibility of seven lines cannot happen, so it suffices to argue that 15 lines cannot occur as well. By \autoref{prop:action-preserves-type}, we have that the action preserves topological type. Since we only have two orbits of sizes 3 and 12, we see that we cannot possibly have 9 hyperbolic lines and 6 elliptic lines, which are the prescribed types via Segre's theorem, hence we cannot have 15 lines. To argue existence of the other solutions, we observe that the Fermat cubic is an example of a symmetric real cubic surface with three lines, while the Clebsch is a symmetric real cubic surface admitting all 27.
\end{proof}

\printbibliography
\end{document}